\documentclass[11pt,reqno]{amsart}
\usepackage{amssymb,mathrsfs,color,mathtools,empheq, verbatim, epstopdf}
\usepackage{pinlabel}
\mathtoolsset{showonlyrefs}
\usepackage{hyperref} 

\usepackage{enumerate, bbm}
\usepackage{tensor}
\usepackage{graphicx}
\usepackage{xcolor} 
\usepackage{tensor}
\usepackage{slashed}
\usepackage{cite}

\usepackage[shortlabels]{enumitem}

\usepackage{geometry}

\numberwithin{equation}{section}

\newtheorem{mainthm}{Theorem}
\newtheorem{thm}{Theorem}[section]
\newtheorem{cor}[thm]{Corollary}
\newtheorem{lem}[thm]{Lemma}
\newtheorem{prop}[thm]{Proposition}

\newtheorem{con}{Conjecture}

\theoremstyle{definition} 
\newtheorem{rem}[thm]{Remark}
\newtheorem{defn}[thm]{Definition}

\theoremstyle{remark}

\renewcommand{\div}{\operatorname{div}}

\definecolor{deepgreen}{cmyk}{1,0,1,0.5}

\newcommand{\R}{\mathbb{R}}

\newcommand{\ga}{\gamma}

\newcommand{\Lam}{\Lambda} 
\newcommand{\Sig}{\Sigma}

\makeatletter

\newcommand{\Rmnum}[1]{\expandafter\@slowromancap\romannumeral #1@}
\makeatother

\newcommand{\lec}{\lesssim}
\newcommand{\gec}{\gtrsim}

\newcommand{\ti}{\widetilde}

\newcommand{\EQ}[1]{\begin{equation}\begin{split} #1 \end{split}\end{equation}}

\newcommand{\Del}[1]{}

\definecolor{green}{rgb}{0,0.8,0} 

\newcommand{\alp}{\alpha}

\newcommand{\veps}{\varepsilon}

\newcommand{\bfn}{{\bf n}}

\newcommand{\bfC}{{\bf C}}
\newcommand{\bfD}{{\bf D}}

\newcommand{\bbN}{\mathbb N}

\newcommand{\bbR}{\mathbb R}
\newcommand{\bbS}{\mathbb S}

\newcommand{\calA}{\mathcal A}
\newcommand{\calB}{\mathcal B}
\newcommand{\calC}{\mathcal C}

\newcommand{\calE}{\mathcal E}
\newcommand{\calF}{\mathcal F}
\newcommand{\calG}{\mathcal G}

\newcommand{\calK}{\mathcal K}
\newcommand{\calL}{\mathcal L}
\newcommand{\calM}{\mathcal M}
\newcommand{\calN}{\mathcal N}

\newcommand{\calQ}{\mathcal Q}

\newcommand{\calS}{\mathcal S}
\newcommand{\calT}{\mathcal T}
\newcommand{\calU}{\mathcal U}

\title{On the Calabi--Yau Conjectures for Minimal Hypersurfaces in Higher Dimensions}
\author{Shrey Aryan and Alexander D. McWeeney}
\begin{document}
\maketitle
\begin{abstract}
In this paper, we study the Calabi--Yau conjectures for complete minimal hypersurfaces $\Sigma^{n}\subset \mathbb{R}^{n+1}$ in dimensions $n\ge 3$. These conjectures ask whether a complete minimal hypersurface must be unbounded, and more strongly whether it must be proper. For the unboundedness question, we prove a chord--arc estimate for an embedded minimal disk with bounded curvature, showing that intrinsic distance is controlled by a polynomial of the extrinsic distance. On the other hand, using gluing techniques, we construct a complete, improperly embedded minimal hypersurface in $\mathbb{R}^{n+1}$ for every $n\ge 3$. This example shows that the properness conjecture suggested by the deep work of Colding--Minicozzi~\cite{CM10} in the case $n=2$ fails in higher dimensions.
\end{abstract}
\section{Introduction}

In 1965, E. Calabi \cite{Ca} proposed the following conjecture:
\begin{con}\label{con:1}
Let $n\geq 2$. A complete minimal hypersurface $\Sigma^{n}\subset \mathbb{R}^{n+1}$ must be unbounded.
\end{con}
When $n=2$, Nadirashvili \cite{Na1} constructed a complete minimal immersion of a disk into the unit ball, thus disproving the conjecture for minimal immersions in $\mathbb{R}^3$. Conjecture \ref{con:1} also appeared in \cite[pg.~212]{chern1966geometry} and was later highlighted in Yau’s survey \cite[pg.~360]{Ya2}, where, referring to Nadirashvili's construction, Yau asked what geometric properties such examples must satisfy. In particular, Yau asked whether any such bounded minimal surface can be embedded. This was addressed in \cite{CM10}, where Colding--Minicozzi proved the Calabi--Yau conjectures by showing that any complete embedded minimal surface in $\mathbb{R}^3$ with finite topology is proper and therefore unbounded. It is therefore natural to ask what remains true in higher dimensions, that is, for $n\ge 3$.

Our first main result proves Conjecture \ref{con:1} under a uniform curvature bound, via a chord--arc estimate for embedded minimal disks:
\begin{mainthm}\label{thm:chordarc}
    Let $3\leq n\leq 5$, \(\Sigma^n \subset \mathbb{R}^{n+1}\) be a two-sided, embedded minimal disk with \(\sup_{\Sigma}|A| < \infty\).
    Then there exists a constant \(c = c\bigl(n,\sup_{\Sigma}|A|\bigr) > 0\) with the following property: for every \(x\in \Sigma\) and every \(R>1\) such that the intrinsic ball \(\mathcal{B}_{R}(x)\subset \Sigma\), the connected component \(\Sigma_{x, R}\) of the Euclidean ball \(B_{R}(x)\cap \Sigma\) containing \(x\) satisfies
    \[
        \Sigma_{x, R}\subset \mathcal{B}_{cR^n}(x).
    \]
    In particular, intrinsic distances are quantitatively controlled by extrinsic distances on \(\Sigma_{x, R}\), with constants depending only on the dimension and the curvature bound.
\end{mainthm}

The second part of the paper addresses the stronger version of the Calabi--Yau conjecture in higher dimensions, namely properness. While properness holds for complete embedded minimal surfaces in \(\mathbb{R}^3\) by \cite{CM10}, we show that the natural higher-dimensional analogue fails in general:
\begin{mainthm}\label{thm:nonproper}
For every \(n\ge 3\), there exists a complete, embedded, minimal hypersurface \(\Sigma_\infty^{n}\subset \mathbb{R}^{n+1}\) that is not proper.
\end{mainthm}

The restrictions on the dimensions in Theorem \ref{thm:chordarc} arise from curvature estimates coming from the recent resolution of the Stable Bernstein conjecture in dimensions $3\leq n \leq 5$. Since Theorem \ref{thm:chordarc} implies unboundedness for complete embedded minimal hypersurfaces with bounded curvature in dimensions $3\leq n\leq 5$, we also include, for completeness, a proof of unboundedness in all dimensions $n\geq 3$. While this fact seems well known to experts (cf.\ \cite{Carlotto_2016} and the proof of Theorem 2 in \cite{rosenberg2001intersection}), we have not found a detailed argument in the literature that applies uniformly to all dimensions.
\begin{mainthm}\label{thm:unbounded}
Let \(\Sigma^{n} \subset \mathbb{R}^{n+1}\) be a complete, embedded minimal hypersurface and assume that \(\sup_{\Sigma}|A|^2 < \infty\). Then \(\Sigma\) is unbounded.
\end{mainthm}

Extending the above result to show properness for complete embedded minimal hypersurfaces with bounded curvature seems nontrivial, in part due to the failure of the half-space theorem. We expect that with suitable topological assumptions one could get around these issues. However, we can show that for a complete, improperly embedded minimal surface \(\Sigma\), the closure \(\bar \Sigma\) is a minimal lamination, as defined in \cite{CM6} Appendix B with planar limit leaves in ambient dimensions $3\leq n \leq 5.$

At a conceptual level, the proofs of Theorem \ref{thm:chordarc} and Theorem \ref{thm:unbounded} follow the strategy of \cite{CM10} in the bounded-curvature regime: when two embedded minimal sheets are very close but disjoint, their separation function satisfies a uniformly elliptic equation whose coefficients are controlled by \(|A|\); a Harnack inequality then yields quantitative control on how fast the sheets can separate.

Theorem \ref{thm:nonproper} is inspired by the work of Fakhi and Pacard \cite{Fakhi-Pacard} and Kapouleas \cite{MR1601434}. Fahki and Pacard construct complete, immersed minimal surfaces in \(\mathbb{R}^{n+1}\), \(n \ge 3\), with finite total curvature and arbitrarily many ends, and Kapouleas produces complete, embedded minimal surfaces in \(\mathbb{R}^3\) with finite total curvature and arbitrarily many ends. These results have also been generalized by Coutant in \cite{coutant2012deformation} to higher dimensions where, amongst other things, the author constructs examples of complete, embedded minimal hypersurfaces in \(\mathbb{R}^{n+1}\) with finite total curvature and arbitrarily many ends. In our work, we modify Fahki and Pacard's construction to produce embedded examples:
\begin{mainthm}\label{thm:k-ends}
For \(n\ge 3\) and \(k \ge 1\), there exists a complete, embedded, minimal hypersurface \(\Sigma_k^{n}\subset \mathbb{R}^{n+1}\) that has finite total curvature and \(k\) ends.
\end{mainthm}
Notably, both Kapouleas \cite{MR1601434} and Fahki-Pacard \cite{Fakhi-Pacard} produce their examples by gluing catenoidal necks. We proceed similarly by iteratively gluing catenoids to form a sequence of surfaces \(\Sigma_k^n\) with \(k\)-ends. To prove Theorem \ref{thm:nonproper}, we will construct such a sequence \(\Sigma_k^n\) that converges locally smoothly to a hypersurface \(\Sigma_\infty^n\) that is complete, embedded, and improper. To our knowledge, this gives the first example of an improperly embedded minimal hypersurface in $\R^{n+1}$. The constructed hypersurface has unbounded curvature, infinite genus, and infinite ends. In particular, it shows that the properness conjecture fails for embedded minimal hypersurfaces of infinite topology in higher dimensions. By contrast, in the surface case $n=2$, the results of \cite{CM10, meeks2021embedded} indicate that properness for embedded minimal surfaces should hold under the additional assumption of finite genus and therefore it would be interesting to find an analogue of Theorem \ref{thm:nonproper} when $n=2.$

Finally, we direct the reader to the survey \cite{MR4415896} by Breiner, Kapouleas, and Kleene for an overview on the state of gluing constructions for constant mean curvature surfaces.

\subsection{Acknowledgments}
The authors would like to thank their advisors Tobias Colding and William Minicozzi for many helpful discussions and encouragement. We are also grateful for stimulating discussions with Otis Chodosh, David Jersion and Peter McGrath. The first author acknowledges support from the Simons Dissertation Fellowship and the second author acknowledges support from the National Science Foundation.

\section{Chord--Arc Estimates for Bounded Curvature}\label{sec:chord-arc}
The main goal of this section is to prove Theorem \ref{thm:chordarc} and Theorem \ref{thm:unbounded}.

Let \(\Sigma \subset \bbR^n\) be an embedded minimal hypersurface with bounded curvature, i.e. \(\sup_\Sigma |A| < \infty\), and let \(x \in \Sigma\). For \(R > 0\), we let \(B_R(x)\) denote an extrinsic ball of radius \(R\) in \(\bbR^{n+1}\), and we let \(\calB_R(x)\) denote an intrinsic ball of radius \(R\) in \(\Sigma\). We denote by \(\Sigma_{x, R}\) the component of \(\Sigma \cap B_R(x)\) containing \(x\).
Furthermore, since we assume $\Sigma$ to have bounded curvature, we can find a graphical radius as follows:
\begin{lem}\label{graphical radius and chord arc for graphical radius}
    Let \(\Sigma^{n} \subset \mathbb{R}^{n+1}\) be an embedded minimal surface with bounded curvature \(\sup_{\Sigma}|A| \le C_A\). Then there is a number \(0 < R_{\Sigma} < 1/C_A\) such that for all \(x \in \Sigma\) and \(R \le R_{\Sigma}\), the following properties hold:
    \begin{enumerate}
        \item The intrinsic ball \(\mathcal{B}_{R}(x)\) is graphical.
        \item \(\calB_R(x)\) can be written as a graph of a function \(u: \Omega \subset T_x\Sigma \rightarrow \mathbb{R}^{n + 1}\) such that \(|\nabla u| < 1\) and \(|\nabla^2 u| < c(C_A)\).
        \item There exists a \(\delta_c > 0\) depending only on dimension and the curvature bound \(C_A\) such that
        \begin{equation}\label{eq: graphical weak chord arc}
            \Sigma_{x, \delta_cR} \subset \mathcal{B}_{R/2}(x).
        \end{equation}
    \end{enumerate}
    Importantly, the number \(R_\Sigma\) depends only on the curvature bound; in particular, this means that \(\calB_{R_\Sigma}(x)\) may contain points in \(\partial \Sigma\).
    
\end{lem}
\begin{proof}
    This result and its proof are essentially the same as Lemma 2.4 in \cite{CM1}. Using the Gauss map \(N: \Sigma \rightarrow \bbS^n\) and the uniform bound on \(|A|\), we can find a uniform radius on which the unit normal of \(\Sigma\) lies above the equator of \(\mathbb{S}^{n}\). We also have a bound \(|\nabla N| < C(n)|A|\). Integrating \(\nabla N\) now produces the bound \(|\nabla u| < 1\) for sufficiently small \(R_\Sigma\) (depending on \(|A|\)), and also produces the \(\delta_c\) in equation (\ref{eq: graphical weak chord arc}).
\end{proof}

Our goal is to prove a statement along the lines of \(\Sigma_{x, R_0} \subset \calB_{R_1}(x)\), for any \(R_0\) (not just \(R_0\) smaller than the graphical radius). To do so, we will proceed by contradiction in the following manner:

\paragraph{\textit{Step 1}} First, we assume that \(\Sigma_{x, R_0} \cap \calB_{R}(x) \neq \emptyset\) for all \(R > 0\). It will follow that there exists \(z_i \in \Sigma_{x, R_0} \cap \partial\calB_{i}(x)\) for all \(i \in \bbN\).\\

\paragraph{\textit{Step 2}} The sequence \(z_i\) is contained in \(B_{R_0}(x)\), and so it is Cauchy. Since \(\Sigma\) is embedded and has bounded curvature, this implies that the minimal disks \(\calB_{R_\Sigma}(z_i)\) are nearly parallel to each other for large \(i\). Roughly speaking, \(\Sigma_{x, R_0}\) must fold on itself infinitely many times.\\

\paragraph{\textit{Step 3}} Using that \(\Sigma\) is embedded and that the disks \(\calB_{R_\Sigma}(z_i)\) are graphical, we will show that the separation function \(u\) between \(\calB_{R_\Sigma}(z_i)\) and \(\calB_{R_\Sigma}(z_{i + 1})\) satisfies an elliptic PDE and a Harnack inequality. This will lead to estimates showing that when the disks \(\calB_{R_\Sigma}(z_i)\) are close to each other, they are almost flat and move away from each other at a bounded rate. The closer the disks are, the slower they move away from each other, and the flatter they are.\\

\paragraph{\textit{Step 4}} By iterating the Harnack inequality for the separation function, we will show that for some large \(I\), there is a disk \(\calB_{R_I}(z_I)\) for some large \(R_I \gg R_0\) which is almost flat and which escapes \(B_{R_0}(x)\), a contradiction. In other words, the places where \(\Sigma_{x, R_0}\) folds on itself are too wide and too flat to stay contained in the extrinsic ball \(B_{R_0}\).\\

We now proceed to the detailed proofs. As mentioned above, we will use an elliptic PDE describing nearby minimal sheets to control their rate of separation. We begin with the following Lemma.
\begin{lem}\label{PDE used to get the Harnack inequality}
Let $\Sigma^n\subset \R^{n+1}$, $n\geq 3$ be a minimal hypersurface, and let \(u: \Sigma \rightarrow \bbR\) be a function such that the graph of \(u\) over \(\Sigma\) is also a minimal hypersurface. By graph of \(u\) over \(\Sigma\), we mean the set \(\{x + u(x)\bfn_{\Sigma}(x) | x \in \Sigma\}\), where \(\bfn_\Sigma\) is a choice of normal vector of \(\Sigma\). If $\max \{|u||A|,|\nabla u|\} \leq 1$, then $u$ satisfies the equation
\begin{align}\label{eqn:minimal-graph}
\operatorname{div}[(I+\bar{L}) \nabla u]+(1+Q)|A|^2 u+Q_{i j} A_{i j}=0,    
\end{align}
where $\bar{L} \leq C(|u||A|+|\nabla u|),|Q| \leq C(|u||A|+|\nabla u|)^2$, and $\left|Q_{i j}\right| \leq C(|u||A|+$ $|\nabla u|)^2$.    
\end{lem}
\begin{proof}
We follow the argument in the proof of Lemma 2.26 in \cite{CM1}. The key point is that the first-order term arising in the computation of the determinant of the metric $g$ for the graph vanishes because $\Sigma$ is minimal. Let $x$ be the position vector, $N$ the unit normal for $\Sigma$, and let $\left\{e_1, e_2,\cdots,e_{n}\right\}$ be an orthonormal frame along $\Sigma$. We will extend $u, e_1, e_2,\cdots,e_{n}$, and $N$ to a small neighborhood of $\Sigma$ by making all of them constant in the normal direction.  Set $A_{i j}=\left\langle A\left(e_i, e_j\right), N\right\rangle$, so that at any point in $\Sigma$ we have
$$
\nabla_{e_i} N=-A_{i j} e_j .
$$
The graph of $u$ over $\Sigma$ is given by
$$
x \rightarrow X(x)=x+u(x) N,
$$
so the tangent is spanned by $X_1,\cdots, X_{n}$ where
$$
X_i=\nabla_{e_i} X=e_i+u_i N-u A_{i k} e_k .
$$
The metric $g_{i j}$ for the graph satisfies
$$
\begin{aligned}
g_{i j} & =\left\langle X_i, X_j\right\rangle=\delta_{i j}+u_i u_j+u^2 A_{i k} A_{j k}-2 u A_{i j}.
\end{aligned}
$$
Since we will consider solutions where $|\nabla u|+|u||A|$ is small, we may write
\begin{align}\label{eqn:metric}
g_{i j}=\delta_{ij}-2 u A_{i j}+Q_{i j},
\end{align}
where $Q_{i j}$ denotes a matrix that is quadratic in $|\nabla u|$ and $|u||A|$ (and will be allowed to vary from line to line). More precisely, there exists $C$ so that if $\max \{|u||A|,|\nabla u|\} \leq 1$, then
$$
\left|Q_{i j}(x)\right| \leq C(|u||A|+|\nabla u|)^2 .
$$
It follows that the inverse metric is given by
$$
g^{i j}=\delta_{i j}+2 u A_{i j}+Q_{i j} .
$$
Using \eqref{eqn:metric} we get
$$
\operatorname{det} g_{i j}=\Pi_{i=1}^{n}\left(1-2 u A_{ii}\right)+Q=1+Q
$$
where the second equality used minimality (i.e., $\sum_{i=1}^{n}A_{ii}=0$ ) and we use the same $Q$ denote terms that are quadratic in $|u||A|$ and $|\nabla u|$. Whenever $J(s)$ is a differentiable path of matrices, the derivative at 0 of the determinant is given by
$$
\left.\frac{d}{d s}\right|_{s=0} \operatorname{det} J(s)=\operatorname{det} J(0) \operatorname{Trace}\left(J^{-1}(0) J^{\prime}(0)\right) .
$$
Applying this with $J(s)=g_{i j}(s)$ where $g_{i j}(s)$ is computed with $(u+s v)$ in place of $u$, we get
$$
\begin{aligned}
\left.\frac{d}{d s}\right|_{s=0}  \operatorname{det} g_{i j} 
& =(1+Q)\left(\delta_{i j}+2 u A_{i j}+Q_{i j}\right)\left(u_i v_j + u_j v_i +uv A_{i k} A_{j k}-2v A_{i j}\right) \\
& =2\langle\nabla u, \nabla v\rangle -2(1+Q)|A|^2 uv  +Q_{i j} A_{i j} v+\langle\bar{L} \nabla u, \nabla v\rangle,
\end{aligned}
$$
where in the second equality we note that 
\begin{align}
\left.J'(0) = \frac{d}{ds}\right|_{s=0} g_{ij}(s) = u_i v_j + u_j v_i +uv A_{i k} A_{j k}-2v A_{i j}.
\end{align}
and the last term $\bar{L}$ is a matrix that is linearly bounded in $|\nabla u|+|A||u|$ (again assuming that $\max \{|u||A|,|\nabla u|\} \leq 1$ ) and we also used $(1+Q) Q_{i j}=$ $Q_{i j}$. Hence, if the graph of $u$ is minimal, then we get for every $v$ with compact support that
$$
0=\int_{\Sigma}\langle(I+\bar{L}) \nabla u, \nabla v\rangle-(1+Q)|A|^2 u v+Q_{i j} A_{i j} v .
$$
Integrating by parts, implies that $u$ satisfies \eqref{eqn:minimal-graph}.
\end{proof}
We now give a condition where the bound \(\max\{|u||A|, |\nabla u|\} \le 1\) holds.
\begin{lem}\label{lem: lemma  2.11 in CM}
    Let \(\Sigma_1, \Sigma_2\) be disjoint minimal disks with bounded curvature \(\sup_{\Sigma_i}|A| < C_A\). Let \(R_\Sigma = \min\{R_{\Sigma_1}, R_{\Sigma_2}\}\) be the smaller of the graphical radii of \(\Sigma_1\) and \(\Sigma_2\). Suppose \(x_1 \in \Sigma_1\), \(x_2 \in \Sigma_2\), and \(\calB_{R_\Sigma/2}(x_i) \cap \partial \Sigma_i = \emptyset\). Then there exists \(\veps_0 = \veps_0(n, C_A) > 0\) and \(r_0 = r_0(n, C_A) > 0\) such that if \(|x_1 - x_2| < \veps_0\), then \(\Sigma_1 \cap B_{r_0}(x_1)\) and \(\Sigma_2 \cap B_{r_0}(x_2)\) are both graphs over the same plane and there exists a function \(u: \calB_{r_0}(x_1) \subset \Sigma_1 \rightarrow \bbR\) such that \(\{x + u(x)\bfn_{\Sigma_1}(x), x \in \calB_{r_0}(x_1)\} \subset \Sigma_2\) and \(|\nabla u| + |u||A| < 1\).
\end{lem}
This lemma roughly says that if \(\Sigma_1\) and \(\Sigma_2\) are nearby, disjoint minimal disks, then one can be written as a normal graph over the other.
\begin{proof}
    First, let \(R_\Sigma\) be a graphical radius valid for both \(\Sigma_1\) and \(\Sigma_2\) (depending only on \(C_A\)), and choose \(R\) small enough that \(\calB_{R}(x_i) \subset \Sigma_i\). Now Lemma \ref{graphical radius and chord arc for graphical radius} tell us that \(\calB_R(x_i)\) are graphs with bounded gradient and Hessian over \(T_{x_i}\Sigma_i\), with bounds depending only on \(C_A\). If \(|x_1 - x_2|\) is sufficiently small, then \(\bfn_{\Sigma_i}(x_i)\) must be nearly parallel, because if not then bounded curvature would force the disks \(\calB_{R}(x_i)\) to intersect transversely. If \(\bfn_{\Sigma_i}(x_i)\) are sufficiently close, then \(\calB_{R/2}(x_i)\) are graphs over the same plane. The details of the proof follow Lemma 2.11 in \cite{CM4} up to constants depending on the dimension.

    To show the bound on \(|\nabla u| + |u||A|\), we first note that by the gradient bound on \(\calB_{R/2}(x_i)\), \(|u||A| < 1/2\) when \(|u(0)| = |x_1 - x_2|\) is sufficiently small. We will use Proposition 3.2.1 of \cite{zhou2024geometry} to show \(|\nabla u| < 1/2\) in a sufficiently small neighborhood \(\calB_{r_0}(x_1)\), with \(r_0\) depending only on curvature. If \(|x_1 - x_2| < \veps\) is small enough to ensure that \(\calB_{R}(x_i)\) are both graphs, then the gradient bound shows for \(x_i' \in B_{\veps/4}(x_i) \cap \calB_{R}(x_i)\), \(|x_1' - x_2'| < 2\veps\). Thus, \(\calB_{\veps/4}(x_2)\) is in a \(2\veps\)-tubular neighborhood of \(\calB_{\veps/4}(x_1)\) for a sufficiently small \(\veps\) depending only on \(C_A\). Proposition 3.2.1 of \cite{zhou2024geometry} shows that \(|\nabla u|\) is bounded on \(\calB_{\veps/4}(x_1)\) by \(\max\{C_1\sqrt{\veps}, C_2\veps\}\), for constants \(C_1, C_2\) depending only on \(C_A\). Thus, taking \(\veps\) small enough guarantees that \(|\nabla u| < 1\) on \(\calB_{r_0}(x_1)\) where \(r_0 = \veps/4\) and \(\veps\) depends only on \(C_A\).
\end{proof}

The previous two Lemmas will be useful in showing that nearby, but disjoint, minimal hypersurfaces with bounded curvature are $\delta$-stable in the follow sense:
\begin{defn}[$\delta$-stability]
Given $\delta \geq 0$, set
$$
L_{\delta}=\Delta+\left(1-\delta\right)|A|^2 + \operatorname{Ric}_{M}(\nu, \nu)
$$
so $L_0$ is the usual Jacobi operator on $\Sigma$. A domain $\Omega \subset \Sigma$ is $\delta$-stable if
$$
\int \phi L_{\delta} \phi \leq 0
$$
for any compactly supported Lipschitz function $\phi$.
\end{defn}
We will then use curvature estimates for \(\delta\)-stable surfaces to show that nearby sheets of \(\Sigma_{x, R_0}\) must be almost flat. Using the Lemma \ref{PDE used to get the Harnack inequality} and the PDE for the function $\log u$ we deduce that
\begin{lem}\label{lem:1/2-stable}
There exists $\eta>0$ so that if $\Sigma$ is minimal and $u>0$ is a solution of the minimal graph equation over $\Omega \subset \Sigma$ with
$$
|\nabla u|+|u||A| \leq \eta,
$$
then $\Omega$ is $\delta$-stable for \(\delta < 1/2\). 
\end{lem}
\begin{proof}
We follow the argument in the proof of Lemma 2.6 in \cite{CM4}. Set $w=\log u$ and choose a cutoff function $\phi \in C_c^1(\Omega)$. Then Lemma \ref{PDE used to get the Harnack inequality} implies that $w$ solves
\EQ{\label{eqn:w-pde}
\Delta w=-|\nabla w|^2+\operatorname{div}(a \nabla w)+\langle\nabla w, a \nabla w\rangle+\langle b, \nabla w\rangle+(c-1)|A|^2,
}
for functions $a_{i j}, b_j, c$ on $\Sigma$ with $|a|,|c| \leq 3|A||u|+|\nabla u|$ and $|b| \leq 2|A||\nabla u|$. Using Stoke's theorem on the map $\operatorname{div}\left(\phi^2 \nabla w-\phi^2 a \nabla w\right)$, and \eqref{eqn:w-pde} with the bounds $|a|,|c| \leq 3 \eta,|b| \leq 2 \eta|\nabla w|$ yields
\begin{align}
(1-3 \eta) \int \phi^2|A|^2 \leq & -\int \phi^2|\nabla w|^2+\int \phi^2\langle\nabla w, b+a \nabla w\rangle +2 \int \phi(\nabla \phi, \nabla w-a \nabla w\rangle \\
\leq & (5 \eta-1) \int \phi^2|\nabla w|^2+2(1+3 \eta) \int|\phi \nabla w||\nabla \phi|.   
\end{align}
The Lemma then follows from Young's inequality.
\end{proof}
\begin{rem}
Observe that for minimal surfaces $\Sigma^2\subset \R^3$ one can prove a stronger statement; let $\Sigma$ be a minimal surface. If there is a sequence $u_j$ so the graphs of $u_j$ over $\Sigma$ are minimal, the $u_j\neq 0$, and $\left|u_j\right|+\left|\nabla u_j\right| \rightarrow 0$, then $\Sigma$ is stable.
\newline 
Unfortunately, this does not naturally generalize to higher dimensions since elliptic PDEs in two dimensions with bounded coefficients yield holder bounds on the gradients, see Lemma 12.4 in \cite{Gilbarg-Trudinger}. This in general is not true in higher dimensions, however, for our applications the notion of $\delta$-stability is strong enough.    
\end{rem}

The final ingredient needed is based on the recent work on Stable Bernstein Conjecture by \cite{chodosh2023stable,CL, catino2024two,chodosh2024stable,mazet2024stable} and its extension to the $\delta$-stable case by \cite{hong2024delta,cheng2025complete}.
\begin{thm}\label{thm:delta-stable-flat}
For $3 \leq n \leq 5$ and $\delta\in (0,1-\delta_1(n))$, an $n$-dimensional complete two-sided $\delta$--stable minimal hypersurface $\Sigma^n \subset  \mathbb{R}^{n+1}$ is a hyperplane, where $\delta_1(3)=3 / 8$, $\delta_1(4)=2 / 3$ and $\delta_1(5)=21 / 22$.
\end{thm}
By a standard blow up argument, Theorem \ref{thm:delta-stable-flat} implies the following curvature estimate,
\begin{cor}\label{cor:curvature-estimate}
Let $3 \leq n \leq 5$ and \(\delta \in (0, 1-\delta_1(n))\) as in Theorem \ref{thm:delta-stable-flat}. There exists \(C(\delta, n) < \infty\) such that if \(\Sigma^n \subset \mathbb{R}^{n+1}\) is a two-sided, \(\delta\)-stable minimal immersion, then
\[|A_M|d(p, \partial M) \le C(\delta, n).\]
\end{cor}
We now ready to give a proof of Theorem \ref{thm:chordarc}.
\begin{proof}[Proof of Theorem \ref{thm:chordarc}]
For clarity, we will rename \(R\) in the theorem statement to \(R_0\) in this proof. Because \(\Sigma\) is an embedded disk, \(\Sigma_{x, R_0}\) is also an embedded disk by the maximum principle and the monotonicity of topology (cf. \cite{CM1}). Suppose that \(\Sigma_{x, R_0}\) were not contained in an intrinsic ball \(\mathcal{B}_{2NR_{\Sigma}}(x)\) for any integer \(N\), where $R_\Sigma>0$ is defined in Lemma \ref{graphical radius and chord arc for graphical radius}. Then we can choose points \(z_i \in \partial \mathcal{B}_{2R_{\Sigma}i} \cap \Sigma_{x, R_0}\) for every \(i\), as well as a curve \(\sigma \subset \Sigma_{x, R_0}\) which passes through each \(z_i\). This sequence is contained in the ball \(B_{R_0}\), and thus is Cauchy. Choose \(z_{i_1}, z_{i_2} \in \{z_i\}\) with
\[|z_{i_1} - z_{i_2}| < \veps < \veps_0\]
where \(\veps_0\) is produced from Lemma \ref{lem: lemma  2.11 in CM}. For \(r_{\Sigma} < R_\Sigma/2\) the intrinsic balls \(\mathcal{B}_{r_{\Sigma}}(z_{i_j})\) are graphical and disjoint. Reducing \(r_\Sigma\) until it is less than the radius \(r_0\) from Lemma \ref{lem: lemma  2.11 in CM}, the balls \(\mathcal{B}_{r_{\Sigma}}(z_{i_1})\) and \(\mathcal{B}_{r_{\Sigma}}(z_{i_2})\) are both graphs over \(T_{z_{i_1}}\Sigma\), and it is possible to define \(u: \mathcal{B}_{r_{\Sigma}}(z_{i_1}) \rightarrow \mathbb{R}\) such that \(x + u(x)\mathbf{n}_{\Sigma}(x) \in \calB_{r_\Sigma}(z_{i_2})\) for \(x \in \mathcal{B}_{r_{\Sigma}}(z_{i_1})\) with \(|\nabla u| + |u||A| < 1\).

Lemma \ref{PDE used to get the Harnack inequality} now shows that \(u\) satisfies a uniform Harnack inequality on any compact set. We note that because \(\Sigma\) has bounded curvature and the closed balls \(\overline{\mathcal{B}_{r_{\Sigma}}}(x)\) are graphical, the same Harnack constant applies for all such balls.

We now iterate the Harnack inequality. Let \(x_1 = z_{i_1}\) and \(y_1 = z_{i_2}\). We initially had a map
\[x + u(x)\mathbf{n}_{\Sigma}(x): \mathcal{B}_{r_{\Sigma}}(x_1) \rightarrow u(\mathcal{B}_{r_{\Sigma}}(x_1)).\]
Consider now \(x_2 \in \mathcal{B}_{r_{\Sigma}}(x_1)\setminus  \mathcal{B}_{r_{\Sigma}/2}(x_1)\).
By the Harnack inequality,
\[u(x_2) < C|x_1 - y_1| < C\veps.\]
Let \(y_2 = x_2 + u(x_2)\mathbf{n}_{\Sigma}(x_2)\). If \(\mathcal{B}_{R_{\Sigma}}(x_2) \cap \mathcal{B}_{R_{\Sigma}}(y_2) = \emptyset\), then by taking \(\veps\) small enough we can ensure that
\[C\veps < \veps_0.\]
We want to show that by taking \(\veps\) small enough, we can iterate this process arbitrarily many times. Suppose we have iterated the process \(M\) times to produce a finite sequence of points \(x_n \in \Sigma\) and associated points \(y_n = x_n + u(x_n)\mathbf{n}_{\Sigma}(x_n)\) with the properties
\[d_{\Sigma}(x_i, x_j) > r_{\Sigma}/2, \quad |x_i - y_i| < \veps_0.\]
for all \(i \neq j\), with the function \(u\) defined on \(\Sigma_1 = \bigcup_n\mathcal{B}_{r_{\Sigma}}(x_n)\). Suppose furthermore that \(\veps\) is small enough that for \(z \in \Sigma_1\),
\[u(z) < \veps_0.\]
We can do this because
\[0 < u(z) < C^M|x_1 - y_1| < C^M\veps\]
for \(z \in \Sigma_1\) by the Harnack inequality. We claim that \(\mathcal{B}_{R_{\Sigma}}(x_i) \cap \mathcal{B}_{R_{\Sigma}}(y_i) = \emptyset\) for each \(i\). If this is true, then the conditions are met to extend \(u\) to a neighborhood of any point \(x_{n + 1} \in \mathcal{B}_{r_{\Sigma}}(x_i) \setminus \mathcal{B}_{r_{\Sigma}/2}(x_i)\).

Suppose for the sake of contradiction that \(\mathcal{B}_{R_{\Sigma}}(x_i) \cap \mathcal{B}_{R_{\Sigma}}(y_i) \neq \emptyset\). Because \(\Sigma\) is embedded, there is a smooth length minimizing geodesic in \(\mathcal{B}_{R_{\Sigma}}(x_i) \cup \mathcal{B}_{R_{\Sigma}}(y_i)\) connecting \(x_i\) to \(y_i\) (if the geodesic left this set, then it could not be length minimizing). Because \(y_i - x_i\) is nearly parallel to \(\mathbf{n}_{\Sigma}(x_i)\) and \(\mathbf{n}_{\Sigma}(y_i)\), it follows that \(\gamma'\) is nearly parallel to \(\mathbf{n}_{\Sigma}(x_i)\) and \(\mathbf{n}_{\Sigma}(y_i)\) at some point in \(\mathcal{B}_{R_{\Sigma}}(x_i) \cup \mathcal{B}_{R_{\Sigma}}(y_i)\), but this contradicts the fact that the balls \(\mathcal{B}_{R_{\Sigma}}(x_i), \mathcal{B}_{R_{\Sigma}}(y_i)\) are each graphical.

Now given \(R > 0\), we claim that by choosing \(\veps\) small enough we can choose points \(x_i\) to define \(u\) on a large set \(\Sigma_1 \supset \mathcal{B}_{R}(x_1)\) (note that \(x_1\) will depend on \(\veps\), but the critical point is the radius \(R\)). Because the set \(\mathcal{B}_{R}(x_1)\) is compact, it can be covered by finitely many graphical intrinsic balls \(\mathcal{B}_{r_{\Sigma}}(p_i)\) with \(d_{\Sigma}(p_i, p_j) > r_{\Sigma}/2\). This number of course does not depend on \(x_1\), and so we need only choose \(\veps\) small enough to repeat the Harnack inequality enough times.

As discussed above, by choosing \(\veps\) and \(r_\Sigma\) sufficiently small, we have ensured that the conditions of Lemma \ref{lem:1/2-stable} hold. It follows that \(\mathcal{B}_{R}(x_1)\) is \(\delta\)-stable for some \(\delta \in (0, 1-\delta_1(n))\). Thus Corollary \ref{cor:curvature-estimate} on \(\mathcal{B}_{R}(x_1)\) implies
\[\sup_{\mathcal{B}_{R/2}(x_1)}|A| \le \frac{2C(\delta,n)}{R}.\]
Thus by taking \(R > 10C(\delta)R_0\), we can ensure that \(\Sigma_{x_1, 10R_0}\) is graphical. By potentially taking \(R > 10C(\delta)MR_0\) for a sufficiently large \(M \ge 1\), we can ensure that
\[B_{5R_0}(x_1) \cap \partial \mathcal{B}_{10R_0}(x_1) = \emptyset.\]
However, we initially had \(x_1 = z_I\) for some \(I\), and \(z_{I + J} \in B_{R_0}(0) \cap \partial \mathcal{B}_{2(I + J)R_{\Sigma}}(0)\) by construction. Choosing \(J\) so that \(2JR_{\Sigma} > 10R_0\), it follows that the segment of the curve \(\sigma\) connecting \(z_I\) and \(z_{I + J}\) would have to leave \(B_{R_0}(0)\). This is because \(B_{R_0}(0) \subset B_{5R_0}(x_1)\), and the intrinsic distance between \(z_I \in \partial\mathcal{B}_{2IR_{\Sigma}}\) and \(z_{I + J} \in \partial\mathcal{B}_{2(I + J)R_{\Sigma}}(0)\) is at least \(2JR_{\Sigma} > 10R_0\), meaning that \(\sigma\) would have to cross \(\partial \mathcal{B}_{10R_0}(x_1)\). This contradiction completes the proof.

We have shown that given \(R_0>0\), there exists an \(R_1\) such that
\begin{equation}
    \Sigma_{x, R_0} \subset \calB_{R_1}(x).
\end{equation}
Now we want to show that \(R_1\) is a polynomial of \(R_0\). Fix \(R_1>0\) such that
\begin{equation}
    \Sigma_{x, 1} \subset \calB_{R_1}(x).
\end{equation}
Let \(B_{1}(x_i)\) be a cover of \(B_{R}(x)\) consisting of \(2R^n\) balls (we emphasize that these are Euclidean balls in \(\mathbb{R}^n\), which is why we know we can take such a cover). Then, by the chord-arc property and the triangle inequality,
\begin{equation}
\Sigma_{x, R}  \subset \Sigma_{x, R} \cap \bigcup_{i = 1}^{2R^n}B_1(x_i) \subset \Sigma_{x_1', 1} \cup \cdots \cup \Sigma_{x_{2R^n}', 1}  \subset \calB_{R_1}(x_1') \cup \cdots \cup \calB_{R_1}(x_{2R^n}')  \subset \calB_{4R_1R^n}(x)
\end{equation}
where $x'_i\in B_{1}(x_i)\cap \Sigma_{x,R}.$
\end{proof}

As a consequence of the chord-arc estimate, we can show that a complete, embedded minimal hypersurfaces with bounded curvature $\Sigma^n \subset \R^{n+1}$ for $3\leq n \leq 5$ are unbounded. However, this fact can be proved in any dimension $n\geq 2$ by a more direct argument, which might be known to experts but since we could not locate it in the literature we write it here for completeness. 
\begin{lem}\label{lem:unbounded}
Let $n\geq 2$ and \(\Sigma^{n} \subset \mathbb{R}^{n+1}\) be a complete, embedded minimal {hypersurface} and assume that \(\sup_{\Sigma}|A| \le C_A\) for some \(C > 0\). Then \(\Sigma\) is unbounded.
\end{lem}
\begin{proof}
Let \(\Sigma \subset \mathbb{R}^n\) be a complete embedded minimal surface with bounded curvature, say \(\sup_{\Sigma}|A| \le C_A\). Suppose that \(\Sigma\) is bounded. Then the closure \(\overline{\Sigma}\) is compact and thus there is a smallest radius \(R_0\) such that \(\Sigma \subset \overline{\Sigma} \subset B_{R_0}(0)\). Let \(z \in \overline{\Sigma} \cap \partial B_{R_0}(0)\).

There is a sequence of points \(z_i \in \Sigma\) approaching \(z\). We claim that by passing to a subsequence, we can assume that the sequence is unbounded and in particular that \(d_{\Sigma}(z_i, z_j) > R_{\Sigma}/2\). If not, then the sequence \(\{z_i\}\) would be bounded, and because \(\Sigma\) is complete it would converge to a point \(z \in \Sigma\). But then \(z\) would be a maximum of \(|\cdot|^2\) on \(\Sigma\), which violates the maximum principle.

Because curvature is bounded and \(\Sigma\) is embedded, we can pass to another subsequence so that the normals \(\mathbf{n}_{\Sigma}(z_i)\) to \(\Sigma\) at the \(z_i\) converge to the normal \(\mathbf{N}\) of \(\partial B_{R_0}(0)\) at \(z\). We may assume after a rotation of \(\mathbb{R}^n\) that \(z\) is at the origin and that \(\mathbf{N}\) is parallel to the \(x_n\) axis.

Fix \(r_{\Sigma} < \delta_c R_{\Sigma}\). Let \(D_{r_{\Sigma}} = B_{r_{\Sigma}}(0) \times \{x_n = 0\}\) and consider the cylinder \(D_{r_{\Sigma}}(0) \times \mathbb{R}\). Let \(\Sigma_{i}\) be the component of \(\Sigma \cap D_{r_{\Sigma}}(0) \times \mathbb{R}\) containing \(z_i\). By the uniform graph radius, we know that the \(\Sigma_i\) are disjoint and can be written as the graphs of functions \(u_i: D_{r_{\Sigma}}(0) \rightarrow \mathbb{R}\) with \(u_i(0) \rightarrow 0\). By the uniform curvature bound, the functions \(u_i\) are uniformly bounded, and because \(\Sigma\) is minimal we know that the \(u_i\) solve the minimal surface equation. Thus \(u_i \in C^2(D_{r_{\Sigma}/2})\), and corollary 16.7 in \cite{Gilbarg-Trudinger} shows that the derivatives of the \(u_i\) of each order have a uniform bound (depending on the order). It follows that \(u_i \rightarrow u\) in \(C^2(D_{r_{\Sigma}/4}(0))\). The limit function \(u\) satisfies the minimal surface equation by continuity, and so its graph is a minimal surface.

By our original construction, the surfaces \(\Sigma_i\) lie above the plane \(\{x_n = 0\}\), because before the rotation of \(\mathbb{R}^n\) they were in the interior of \(B_{R_0}(0)\). By the maximum principle, the graph of \(u\) must lie entirely in the plane \(\{x_n = 0\}\) or go below the plane \(\{x_n = 0\}\), since the graph is a minimal surface and \(u(0) = 0\). But then it follows that some \(\Sigma_i\) leaves the ball \(B_{R_0}\), a contradiction. This completes the proof.
\end{proof}
Finally, while we cannot prove that a complete embedded minimal hypersurface $\Sigma^n\subset \R^{n+1}$ with bounded curvature is proper, we observe that $\bar{\Sigma}$ is a minimal lamination of $\R^{n+1}$. Recall that  a minimal lamination of a smooth manifold $M^{n+1}$ is defined as follows (cf. Appendix B of \cite{CM6})
\begin{defn}[Minimal lamination]\label{def:minimal-lamination}
Let $M^{n+1}$ be an $(n+1)$-dimensional smooth manifold.
A \emph{codimension-one lamination} of $M$ is a collection $\mathcal L$ of smooth, disjoint $n$-dimensional submanifolds of $M$ (called \emph{leaves})
such that:
\begin{enumerate}
  \item The union $\bigcup_{\Lambda\in\mathcal L}\Lambda$ is a closed subset of $M$.
  \item For every $x\in M$ there exists an open neighborhood $U\ni x$ and a coordinate chart
  $\Phi:U\to \Phi(U)\subset \mathbb R^{n+1}$ with the property that, in these coordinates,
  the leaves of $\mathcal L$ pass through $\Phi(U)$ in slices of the form $(\mathbb R^n\times\{0\})\cap \Phi(U)$.
\end{enumerate}
A \emph{minimal lamination} is a lamination whose leaves are minimal. A \emph{limit leaf} is a leaf $\Lambda$ of $\calL$ that is contained in the closure of $\calL\setminus \Lam$ (cf. Definition 4.3.2 in \cite{zhou2024geometry}).
\end{defn}
Furthermore, by the recent resolution of the Stable Bernstein Conjecture in $\R^{n+1}$ for $3\leq n \leq 5$ we can now show that the limit leaves of $\bar{\Sigma}$ are flat, which yields the following trichotomy:
\begin{lem}
Let $3\leq n\leq 5$. Suppose $\Sigma^n\subset \R^{n+1}$ is a complete connected embedded minimal surface in $\mathbb{R}^{n+1}$ with $|A_\Sigma|\leq C_A <\infty$. Then \(\overline{\Sigma}\) is a minimal lamination, and one of the following holds:
\begin{enumerate}
    \item $\Sigma$ is properly embedded in $\mathbb{R}^{n+1}$,
    \item $\Sigma$ is properly embedded in an open halfspace of $\mathbb{R}^{n+1}$ with limit set the boundary plane of this halfspace, or
    \item $\Sigma$ is properly embedded in an open slab of $\mathbb{R}^{n+1}$ with limit set consisting of the boundary planes.
\end{enumerate}
\end{lem}
\begin{proof}
The proof follows from the same argument as in the proof of Lemma 1.1 in \cite{meeks2005uniqueness} where the authors, using the bound on the curvature, first show that $\bar{\Sigma}$ is a minimal lamination of $\R^{n+1}$. Next, by considering the universal cover $\hat{\Lam}$ of any limit leaf $\Lam$ one can show that $\hat{\Lam}$ is stable since any compact domain $\hat{D}$ of $\hat{\Lam}$ is a limit of disjoint minimal domains, which in turn implies that $\hat{D}$ is stable by a contradiction argument. Therefore, $\Lam$ is stable, which by the works of \cite{chodosh2023stable,CL, catino2024two,chodosh2024stable,mazet2024stable} implies that is $\Lam$ is a plane, thus completing the proof of the Lemma.
\end{proof}

\section{Improperly Embedded Complete Minimal Hypersurface in \(\mathbb{R}^{n +1}\)}\label{sec:improper}

The main goal of this section is to prove Theorem \ref{thm:k-ends}, which will be used to prove Theorem \ref{thm:nonproper}. To this end, we will construct complete, embedded, minimal hypersurfaces with $k$ planar ends by gluing \(k\) catenoids together. Heuristically, our surfaces will look like a stack of \(k\) planes, with each pair of planes connected by a catenoidal neck. The width of the neck connecting the \(k\)th and \((k + 1)\)th planes will be decreasing in \(k\), and the necks will be moving away from the origin as $k\to \infty$. The planes will converge to a limit plane, and so the surfaces we produce will have a limit surface which is complete, embedded, and minimal, but improper in $\R^{n+1}$ for all $n\geq 3.$ 

The construction of $\Sigma_k$ is inspired by the work of Fakhi and Pacard \cite{Fakhi-Pacard}, who constructed examples of immersed minimal hypersurfaces in \(\mathbb{R}^{n+1}\) with $k$-planar ends in the same manner. In that work, the authors show that a half catenoid can be glued to a nondegenerate minimal surface (cf. Lemma \ref{lem:invertability-of-L_M}) of finite total curvature. The key idea is to use the scaling of the Green's function in \(\mathbb{R}^n\). When \(n \ge 3\), the Green's function vanishes at infinity, and grows similarly to a catenoid near the origin. Let the Green's function with pole at the origin be denoted \(\gamma_0\) and consider a minimal surface which is the graph of a function \(u\). The graph of \(u + \gamma_0\) will look like a catenoidal neck near the origin and look like the original graph of \(u\) far from the origin. Since \(\gamma_0\) is harmonic, it will be possible to perturb \(u + \gamma_0\) to be a minimal surface; in particular, we can perturb it to match with the graph of a half-catenoid near the origin, gluing the two surfaces.

\subsection{Overview of the Construction}
Since the construction of these embedded minimal hypersurfaces with $k$-planar ends is quite involved, we give an outline. Our contribution is to modify Fakhi and Pacard's \cite{Fakhi-Pacard} gluing process to produce embedded surfaces, and to produce an improperly embedded surface. However, the technical details of the gluing process are still very similar to Fakhi and Pacard's. Much of the details of this section are included to make the paper readable, and to trace our modifications to \cite{Fakhi-Pacard} to ensure they are sound. We emphasize our key modification in the outline.

Let \(\Sigma\) be a given minimal hypersurface and let \(C\) be a half catenoid, which we intend to glue to \(\Sigma\). We achieve the gluing in three steps.\\

\paragraph{\textit{Step 1}} Let \(p_0 \in \Sigma\) be the point at which the gluing will occur. Now we define our initial, unperturbed surfaces. \\

\paragraph{\textit{Step 1(a)}} For \(r_0>0\) chosen sufficiently small, the set \(\Sigma^c_{r_0}=B_{r_0}(p_0)\cap \Sigma\) is a compact graphical portion of \(\Sigma\). By translation and rotation, we may suppose that \(p_0\) is the origin, and that \(\Sigma_{r_0}^c\) is a graph over \(\{x_{n + 1} = 0\}\).

Diverging from Fakhi-Pacard, we do not assume that \(\{x_{n + 1} = 0\}\) is the tangent plane of \(\Sigma\) at \(p_0\). Instead, \(\{x_{n + 1} = 0\}\) can be any plane over which \(\Sigma_{r_0}^c\) is a graph.\\

\paragraph{\textit{Step 1(b)}}We denote by \(\Sigma_{r_0}\) the noncompact piece of \(\Sigma\) that remains after removing \(\Sigma^c_{r_0}\). I.e., \(\Sigma_{r_0} = \Sigma \backslash \Sigma_{r_0}^c\)\\

\paragraph{\textit{Step 1(c)}} We will glue a half catenoid \(C\) to \(\Sigma^c_{r_0}\). We choose \(C\) to be the catenoid oriented with the vertical axis, whose neck has radius 1, and whose boundary \(\partial C\) is a graph of a function over the sphere \(\bbS^{n - 1} \subset \{x_{n+1} = 0\}\).

A priori, the width of the neck of \(C\) may be much larger than \(r_0\). To account for this, we fix \(\veps>0\) sufficiently small, scale \(C\) by the factor \(\veps^{1/(n-1)}\), and then truncate the scaled catenoid so that the neck radius is proportional to \(\veps^{3/(3n-2)}\), thereby obtaining a truncated half catenoid \(C_\veps\). The choice of scaling simplifies later computations.\\
    
\paragraph{\textit{Step 2}} In order to glue \(\Sigma_{r_0}\), \(\Sigma^c_{r_0}\), and \(C_\veps\), we must perturb all 3 surfaces.\\

\paragraph{\textit{Step 2(a)}}We first perturb \(C_\veps\) to produce a surface \(C_\veps(h_\mathrm{II})\) so that \(\partial C_\veps(h_\mathrm{II})\) is the graph of \(h_{\mathrm{II}}\) over the sphere \(r\bbS^{n-1} \subset \{x_{n+1} = 0\}\) (where \(r\) is the radius of the neck). We will not have complete freedom in choosing \(h_\mathrm{II}\) due to the Jacobi fields of \(C_\veps\). Our perturbation will be a graph over \(\{x_{n+1} = 0\}\) near the origin, and away from the origin it will be a normal graph over \(C_\veps\). Importantly, the asymptote plane of \(C_\veps(h_\mathrm{II})\) will be the same as the asymptote plane of \(C_\veps\). This step is identical to Fakhi-Pacard, but we include some proofs for readability.\\

\paragraph{\textit{Step 2(b)}}Next we perturb \(\Sigma^c_{r_0}\). We must glue \(\Sigma^c_{r_0}\) to both \(\Sigma_{r_0}\) and \(C_\veps\). Suppose \(\Sigma^c_{r_0}\) is a graph \((x, u(x))\) over \(B_{r_0} \subset \{x_{n+1} = 0\}\). We let \(\Sigma^c_{r_0, \veps}\) be a graph over \(B_{r_0} \backslash B_{r_\veps/2}\) of
\begin{equation}
    \left(x, u(x) + \frac{\veps}{n - 2}\gamma_0(x)\right)
\end{equation}
where \(\gamma_0\) is a Green's function of \(\Sigma^c_{r_0}\) at the origin and \(r_\veps\) is proportional to the radius of the neck of \(C_\veps\). Now \(\Sigma^c_{r_0, \veps}\) has two boundaries, one which we will glue to \(C_\veps\) and one which we glue to \(\Sigma_{r_0}\). On its inner boundary, \(\Sigma^c_{r_0, \veps}\) looks like a Green's function. The catenoid piece \(C_\veps\) also looks like a Green's function at its boundary, so we will glue the inner boundary of \(\Sigma^c_{r_0, \veps}\) to \(C_\veps\). Since the Green's function on manifolds of dimension \(\ge 3\) goes to \(0\), the outer boundary of \(\Sigma^c_{r_0, \veps}\) looks like the outer boundary of \(\Sigma^c_{r_0}\) (and thus the boundary of \(\Sigma_{r_0}\)). So, we will glue the outer boundary of \(\Sigma^c_{r_0, \veps}\) to \(\Sigma_{r_0}\). In order to make up for degrees of freedom lost due to Jacobi fields, we make some initial geometric transformations to \(\Sigma^c_{r_0, \veps}\) (such as rotations and translations), which we represent by a term \(\calA\). Thus our goal is to produce a perturbation \(\Sigma^c_{r_0, \veps}(h_\mathrm{I}, \calA, h_\mathrm{II})\), where \(h_\mathrm{I}\) is the data on the boundary in common with \(\Sigma_{r_0}\), \(h_\mathrm{II}\) is the data on the boundary in common with \(C_\veps\), and \(\calA\) represents the geometric transformations. This step is where we diverge from Fakhi-Pacard the most. Fakhi-Pacard work with \(\Sigma_{r_0}^c\) as a graph over its tangent plane at the origin; we, on the other hand, allow the plane we are working with to be slightly tilted.\\
            
\paragraph{\textit{Step 2(c)}} Finally, we perturb \(\Sigma_{r_0}\). This is, in general, more difficult than either of our previous perturbations, because \(\Sigma\) is noncompact and does not even have an implicit description. We proceed by assuming that \(\Sigma\) has finite total curvature and planar ends, at which point prior work on perturbing noncompact surfaces \cite{cmc-surface, yamabe-metric} enables us to produce \(\Sigma_{r_0}(h_\mathrm{I})\). This step is almost identical to Fakhi-Pacard, with very minor changes to account for how we handled Step 2(b).\\

\paragraph{\textit{Step 3}} Gluing \(\Sigma_{r_0}(h_\mathrm{I})\), \(\Sigma^c_{r_0, \veps}(h_\mathrm{I}, \calA, h_\mathrm{II})\), and \(C_\veps(h_\mathrm{II})\).\\

\paragraph{\textit{Step 3(a)}} We first define the Cauchy data maps. These maps represent the Cauchy data for each of the boundary value problems we solve, the boundary values being \(h_\mathrm{I}, \calA, h_\mathrm{II}\). Denote \(\calS_\veps(h_\mathrm{II})\) as the Cauchy data for \(C_\veps(h_\mathrm{II})\),  \(\calT_\veps(\calA, h_\mathrm{II})\) as the Cauchy data for the inner boundary problem of \(\Sigma^c_{r_0, \veps}(h_\mathrm{I}, \calA, h_\mathrm{II})\), \(\calU_\veps(h_\mathrm{I}, \calA, h_\mathrm{II})\) as the difference of the Cauchy data for \(\Sigma_{r_0}(h_\mathrm{I})\) and the inner boundary data for \(\Sigma^c_{r_0, \veps}(h_\mathrm{I}, \calA, h_\mathrm{II})\).  We collect these Cauchy data in a map
\begin{equation}
    \bfC_\veps: C^{2, \alpha}(\partial B_{r_0}) \times \mathbb{R}^{2n + 2} \times \pi_\mathrm{II}(C^{2, \alpha}(\bbS^{n-1})) \rightarrow C^{1, \alpha}(\partial B_{r_0}) \times C^{2, \alpha}(\bbS^{n-1}) \times C^{1, \alpha}(\bbS^{n-1})
\end{equation}
(the spaces will be defined over the course of the paper), given by
\begin{equation}
    \bfC_\veps(h_\mathrm{I}, \calA, h_\mathrm{II}) = (\calU_\veps(h_\mathrm{I}, \calA, h_\mathrm{II}), \calT_\veps(\calA, h_\mathrm{II}) - \calS_\veps(h_\mathrm{II})).
\end{equation}
A zero of \(\bfC_\veps\) will indicate that
\begin{equation}
    \Sigma = \Sigma_{r_0}(h_\mathrm{I}) \cup \Sigma^c_{r_0, \veps}(h_\mathrm{I}, \calA, h_\mathrm{II}) \cup C_\veps(h_\mathrm{II})
\end{equation}
is indeed a smooth minimal surface. It is worth noting here that the perturbation via Green's function is what makes it possible for \(\calT_\veps(\calA, h_\mathrm{II}) - \calS_\veps(h_\mathrm{II})\) and \(\calU_\veps(h_\mathrm{I}, \calA, h_\mathrm{II})\) to be small at all.\\
\paragraph{\textit{Step 3(b)}}
Next, we construct a simplified Cauchy data map. In order to find a zero of \(\bfC_\veps\), we construct a simpler map \(\bfC_0\) defined between the same spaces as \(\bfC_\veps\), but which  will be an isomorphism between its domain and range. The simple maps will be Cauchy data for simpler boundary value problems. Denote \(\calS_0(h_\mathrm{II})\) as the simple map for \(C_\veps(h_\mathrm{II})\). It will represent the data for a problem involving the Laplacian instead of the linearized mean curvature operator of \(C_\veps\). We will find
\begin{equation}
    \left\|\left(\mathcal{S}_{\varepsilon}-\mathcal{S}_{0}\right)\left(h_{\mathrm{II}}\right)\right\|_{\mathcal{C}^{2, \alpha} \times \mathcal{C}^{1, \alpha}} \leq c r_{\varepsilon}^{2}.
\end{equation}
Denote \(\calT_0(\calA, h_\mathrm{II})\) as the next simple map, and we again will find
\begin{equation}
    \left\|\left(\calT_\veps-\calT_0\right)\left(h_{\mathrm{II}}\right)\right\|_{\mathcal{C}^{2, \alpha} \times \mathcal{C}^{1, \alpha}} \leq c r_{\varepsilon}^{2}.
\end{equation}
Finally, we have the simple map \(\calU_0(h_\mathrm{I})\) for which
\begin{equation}
    \left\|\mathcal{U}_{\varepsilon}\left(h_\mathrm{I}, \mathcal{A}, h_{\mathrm{II}}\right)-\mathcal{U}_{0}(h_\mathrm{I})\right\|_{\mathcal{C}^{1, \alpha}} \leq c\left(\|h_\mathrm{I}\|_{\mathcal{C}^{2, \alpha}}^{2}+r_{\varepsilon}^{n-2 / 3}\right).
\end{equation}
We collect the simple Cauchy data maps into
\begin{equation}
    \bfC_0 = (\calU_0(h_\mathrm{I}), \calT_0(\calA, h_\mathrm{II}) - \calS_0(h_\mathrm{II})).
\end{equation}
It will happen that \(\bfC_0\) is an isomorphism between its domain and range.\\
\paragraph{\textit{Step 3(c)}}
Our final aim is to find a zero of \(\bfC_\veps\). We reduce to finding a fixed point of
\begin{equation}
    \bfC_0^{-1}(\bfC_0 - \bfC_\veps).
\end{equation}
Noting that \(\bfC_0 - \bfC_\veps\) is small by the bounds we calculated in Step 3(b), we will find a fixed point using the Schauder fixed-point theorem. The operator \(\bfC_0^{-1}(\bfC_0 - \bfC_\veps).\) is however not compact, so we use a sequence of compactly supported smoothing operators and take a limit.\\

The critical difference between our construction and Fakhi-Pacard's is that the asymptote plane of the half-catenoid we glue does not need to be parallel to the tangent plane of \(\Sigma\) at the gluing point. Instead, it can be slightly tilted (in a way we make quantitative later). We now proceed with giving more details to justify each step.
\subsection{Perturbing the catenoid to produce \(C_\veps(h_\mathrm{II})\)}\label{sec:4}
We will now describe how to perturb a half-catenoid to get \(C_\veps(h_\mathrm{II})\). Let \(C\) be the catenoid centered at the origin, oriented vertically along the \(x_{n+1}\) axis, whose neck has radius 1. We denote by \(\rho_0: \mathbb{R} \rightarrow \mathbb{R}\) the function such that \(C\) is the surface of revolution generated by \(\rho_0\). We use a conformal coordinate system
\begin{equation}
    X_0 : \mathbb{R} \times S^{n-1} \rightarrow \mathbb{R}^{n +1}, \quad X_0(s, \theta) = (\phi(s)\theta, \psi(s)).
\end{equation}
where \(\phi,\psi: \mathbb{R} \rightarrow \mathbb{R}\). satisfy the following relations:
\begin{equation}
    \phi = \rho_0 \circ \psi,
\end{equation}
\begin{equation}
    \psi' = \phi^{2 - n}, \quad \psi(0) = 0,
\end{equation}
\begin{equation}
    (\phi')^2 + \phi^{4 - 2n}, \quad \phi(0) = 0.
\end{equation}
The asymptotics of \(\phi\) will be relevant, so we quote
\begin{lem}
    There exists a constant \(A > 0\) so that
    \begin{equation}
        e^{-|s|}\phi(s) = A(1 + O(e^{(2 - 2n)|s|})) \quad \text{as \(s \rightarrow \infty\),}
    \end{equation}
    up to a positive multiplicative constant. In other words, \(\phi(s)\) grows like \(Ae^{|s|}\) for large \(s\).
\end{lem}

We will first perturb the catenoid along its normal vector. We let
\begin{equation}
    N_0(s, \theta) = \frac{1}{\phi(s)}(\psi'(s)\theta, -\phi'(s))
\end{equation}
be a unit normal vector field on \(C\). Consider the surface parametrized by 
\begin{equation}
    X = X_0 + wN_0.
\end{equation}
We have
\begin{prop}\label{prop:3.1}
The hypersurface parameterized by \(X\) is minimal if and only if \(w\) solves the following equation
\begin{equation}\label{eqn3.7}
    \begin{split}
        \mathcal{L}_{0} w & = Q_{2}\left(s, \frac{w}{\phi}, \nabla\left(\frac{w}{\phi}\right), \nabla^{2}\left(\frac{w}{\phi}\right)\right) \\ & +\phi^{n-1} Q_{3}\left(s, \frac{w}{\phi}, \nabla\left(\frac{w}{\phi}\right), \nabla^{2}\left(\frac{w}{\phi}\right)\right)
    \end{split}
\end{equation}
where
\begin{equation}
    \mathcal{L}_{0}=\partial_{s}\left(\phi^{n-2} \partial_{s}\right)+\phi^{n-2} \Delta_{S^{n-1}}+n(n-1) \phi^{-n},
\end{equation}
is the linearized mean curvature operator and
\begin{equation}
    \left(q_{1}, q_{2}, q_{3}\right) \longrightarrow Q_{2}\left(s, q_{1}, q_{2}, q_{3}\right),
\end{equation}
is homogeneous of degree 2 and
\begin{equation}
    \left(q_{1}, q_{2}, q_{3}\right) \longrightarrow Q_{3}\left(s, q_{1}, q_{2}, q_{3}\right),
\end{equation}
consists of higher-order nonlinear terms. In particular, we have
\begin{equation}
    Q_{3}(s, 0,0,0)=0, \quad \nabla_{q_{i}} Q_{3}(s, 0,0,0)=0 \quad \text { and } \quad \nabla_{q_{i} q_{j}}^{2} Q_{3}(s, 0,0,0)=0 .
\end{equation}

Furthermore, the coefficients $Q_{2}$ on the one hand, and the partial derivatives at any order of $Q_{3}$, with respect to the $q_{i}$ 's, computed at any point of some neighborhood $\mathcal{V}$ of \((0,0,0)\) on the other hand, are bounded functions of \(s\) and so are the derivatives of any order of these functions, uniformly in \(\mathcal{V}\).
\end{prop}
Our goal is to solve this equation given boundary data. This is, of course, not possible in general because \(\calL\) has a nontrivial kernel (the Jacobi fields). However, it is possible to find partial inverses on a specific space. We will now describe this construction.
\subsubsection{Inverting \(\calL_0\)}
First, we actually consider a conjugate operator \(\calL = \phi^{(2 - n)/2}\calL_0\phi^{(2-n)/2}\). This gives us the simpler-looking operator
\begin{equation}
    \calL = \partial_{ss}+\Delta_{S^{n-1}}-\left(\frac{n-2}{2}\right)^{2}+\frac{n(3 n-2)}{4} \phi^{2-2 n}.
\end{equation}
We also define the even simpler 
\begin{equation}
    \Delta_0 \equiv \partial_{s s}+\Delta_{S^{n-1}}-\left(\frac{n-2}{2}\right)^2.
\end{equation}
We note that solving \(\calL\) is the same as solving \(\calL_0\). Indeed, by the previous proposition, the hypersurface
\begin{equation}
    X_w = X_0 + w\phi^{\frac{2 - n}{2}}N_0
\end{equation}
is minimal when \(w\) solves
\begin{equation}
    \calL w = \phi^{\frac{2-n}{2}} Q_{2}\left(\phi^{-\frac{n}{2}} w\right)+\phi^{\frac{n}{2}} Q_{3}\left(\phi^{-\frac{n}{2}} w\right).
\end{equation}
To solve \(\calL\), we work on the following weighted function space:
\begin{defn}\label{defn:4.1}
    For all \(\delta \in \mathbb{R}\) and \(S \in \mathbb{R}\), we define the norm
    \begin{equation}
        \|w\|_{k, \alpha, \delta} \equiv \sup _{s \geq S}|e^{-\delta s} w|_{k, \alpha, \left([s, s+1] \times S^{n-1}\right)}.
    \end{equation}
    The norm \(\|\cdot\|_{k, \alpha\left([s, s+1] \times S^{n-1}\right)}\) is the usual Hölder norm. Then
    \begin{equation}
        w\in \mathcal{C}_{\delta}^{k, \alpha}\left([S,+\infty) \times S^{n-1}\right)
    \end{equation}
    if and only if \(\|w\|_{k, \alpha, \delta} < \infty\).
\end{defn}
We will invert \(\calL\) on \(\calC_\delta^{k, \alpha}\) for \(\delta \in \left(-\frac{n + 2}{2}, -\frac{n}{2}\right)\). This choice of \(\delta\) will let the inverse be bounded and unique, as the choice of \(\delta\) guarantees the Jacobi fields of \(C\) are not in \(\calC_\delta^{k, \alpha}\).

Before we proceed, we define some final notation.
\begin{defn}
    Let \(e_{j}(\theta): S^{n-1} \rightarrow \mathbb{R}\), \(j \in \mathbb{N}\) denote the eigenfunctions of \(\Delta_{S^{n-1}}\) with eigenvalues \(\lambda_{j}\). More specifically we have \(\Delta_{S^{n-1}} e_{j}=-\lambda_{j} e_{j}\), and we order the eigenvalues such that \(\lambda_{j} \leq \lambda_{j+1}\). For a function \(\varphi \in L^2(S^{n-1})\) there are constants \(a_j\) so that,
    \begin{equation}
        \varphi = \sum_{j \in \mathbb{N}}a_je_j.
    \end{equation}
    We define projections
    \begin{equation}
        \pi_\text{I}(\varphi) = \sum_{j \le n}a_j e_j \quad \text{and} \quad \pi_\mathrm{II}(\varphi) = \sum_{j \ge n + 1}a_je_j.
    \end{equation}
\end{defn}
It is possible to describe the Jacobi fields of \(C\) in terms of \(e_1, \ldots e_n\). Because of this, a maximum principle for \(\calL\) holds only for functions \(w\) which are orthogonal to \(e_1, \ldots, e_n\), and in general we cannot guarantee uniqueness of an inverse based on specified boundary data. However, it turns out we can guarantee uniqueness by specifying boundary data for \(\pi_\mathrm{II} (w)\).

With this in mind, we present propositions for constructing an inverse of \(\calL\). 
\begin{prop}\label{prop:4.3}
    Assume that \(\delta \in\left(-\frac{n+2}{2},-\frac{n}{2}\right)\) and \(\alpha \in(0,1)\) are fixed. For all \(S \in \mathbb{R}\), there is an operator
    \begin{equation}
        \mathcal{G}_{S}: \mathcal{C}_{\delta}^{0, \alpha}\left([S,+\infty) \times S^{n-1}\right) \longrightarrow \mathcal{C}_{\delta}^{2, \alpha}\left([S,+\infty) \times S^{n-1}\right),
    \end{equation}
    such that, for all \(f \in \mathcal{C}_{\delta}^{0, \alpha}\left([S+\infty) \times S^{n-1}\right)\), the function \(w=\mathcal{G}_{S}(f) \in \mathcal{C}_{\delta}^{2, \alpha}\left([S+\infty) \times S^{n-1}\right)\) is the unique solution of
    \begin{equation}
        \begin{cases}\mathcal{L} w=f & \text {in \([S,+\infty) \times S^{n-1}\)} \\ \pi_{\mathrm{II}}(w) = 0 & \text {on \(\{S\} \times S^{n-1}\)}\end{cases} \quad \text{with} \quad w \in \mathcal{C}_{\delta}^{2, \alpha}\left([S+\infty) \times S^{n-1}\right).
    \end{equation}
    Furthermore, there is a constant \(c\) which does not depend on \(\delta\), \(\alpha\), or \(S\), such that
    \begin{equation}
        \|w\|_{2, \alpha, \delta} \leq c\|f\|_{0, \alpha, \delta}.
    \end{equation}
    Finally, if, for each fixed \(s \in[S,+\infty)\), the function \(f(s, \cdot)\) is orthogonal to \(e_{0}, \ldots, e_{n}\) in \(L^2(S^{n-1})\), then so is \(w(s, \cdot)\).
\end{prop}
We also have
\begin{prop}\label{prop:4.5}
    There exists \(c>0\) such that, for all \(S \in \mathbb{R}\) and all \(g_{\mathrm{II}} \in \pi_{\mathrm{II}}\left(\mathcal{C}^{2, \alpha}\left(S^{n-1}\right)\right)\), there exists a unique solution \(w_{0} \in \mathcal{C}_{-\frac{n+2}{2}}^{2, \alpha}\left([S,+\infty) \times S^{n-1}\right)\) to
    \begin{equation}
        \begin{cases}\Delta_{0} w_{0}=0 & \text { in  \((S, \infty) \times S^{n-1}\)} \\ w_{0}=g_{\mathrm{II}} & \text {on \(\{S\} \times S^{n-1}\)}\end{cases} \quad \text{with} \quad w \in \mathcal{C}_{-\frac{n+2}{2}}^{2, \alpha}\left([S+\infty) \times S^{n-1}\right).
    \end{equation}
    Furthermore, we have the bound
    \begin{equation}   
        \left\|w_{0}\right\|_{2, \alpha,-\frac{n+2}{2}} \leq c e^{\frac{n+2}{2} S}\left\|g_{\mathrm{II}}\right\|_{2, \alpha}.
    \end{equation}
\end{prop}
Note that in the last proposition, we did not specify \(\pi_\mathrm{II}(w_0) = g_\mathrm{II}\), as we did in the first. We now combine these propositions to get
\begin{prop}\label{prop:4.6}
    Suppose \(\delta \in\left(-\frac{n+2}{2},-\frac{n}{2}\right)\) and \(\alpha \in (0,1)\) is fixed. Then, for all \(g_\mathrm{II} \in \pi_{\mathrm{II}}\left(\mathcal{C}^{2, \alpha}\left(S^{n-1}\right)\right)\), there is a unique solution \(w = \mathcal{P}_{S}\left(g_{\mathrm{II}}\right) \in \mathcal{C}_{\delta}^{2, \alpha}\left([S,+\infty) \times S^{n-1}\right)\) of
    \begin{equation}
        \begin{cases}\mathcal{L} w=0 & \text { in \((S, \infty) \times S^{n-1}\)} \\ \pi_\mathrm{II}(w)=g_{\mathrm{II}} & \text { on \(\{S\} \times S^{n-1}\)}\end{cases} \quad \text{with} \quad w \in \mathcal{C}_{\delta}^{2, \alpha}\left([S+\infty) \times S^{n-1}\right)
    \end{equation}
Furthermore, we have the bound
\begin{equation*}
    \left\|\mathcal{P}_{S}\left(g_{\mathrm{II}}\right)\right\|_{2, \alpha, \delta} \leq c e^{-\delta S}\left\|g_{\mathrm{II}}\right\|_{2, \alpha}
\end{equation*}
for some constant $c>0$ which is independent of $S$.
\end{prop}
\begin{proof}
    First, let \(w_0 = G_S(g_\mathrm{II})\) solve
    \begin{equation}
        \begin{cases}\Delta_{0} w_{0}=0 & \text { in  \((S,\infty) \times S^{n-1}\)} \\ w_{0}=g_{\mathrm{II}} & \text {on \(\{S\} \times S^{n-1}\)}\end{cases}.
    \end{equation}
    Then, let
    \begin{equation}
        f = -\calL w_0 = -\Delta_0 w_0 - \frac{n(3n - 2)}{4}\phi^{2 - 2n}w_0 = -\frac{n(3n - 2)}{4}\phi^{2 - 2n}w_0.
    \end{equation}
    Then since \(-\frac{n + 2}{2} < \delta < -\frac{n + 2}{2} + 1\), \(\|w_0\|_{2, \alpha, -\frac{n+2}{2}}\) is bounded, and \(\phi\) satisfies the asymptotics \(\phi \rightarrow e^{s}\) as \(s \rightarrow \infty\), we have
    \begin{equation}
        \|f\|_{0, \alpha, \delta} \le C\|w_0\|_{2, \alpha, -\frac{n + 2}{2}} \le ce^{\frac{n+2}{2}S}\|g_\mathrm{II}\|_{2, \alpha}.
    \end{equation}
    It follows that \(\|f\|_{0, \alpha, \delta} < \infty\), so that \(f \in C^{0, \alpha}_\delta([S, \infty) \times \bbS^{n-1})\) and we can find \(w_1 = \calG(f)\) which solves
    \begin{equation}
        \begin{cases}\mathcal{L} w_1 = f  & \text {in \([S,+\infty) \times S^{n-1}\)} \\ \pi_{\mathrm{II}}(w_1) = 0 & \text {on \(\{S\} \times S^{n-1}\)}.\end{cases}
    \end{equation}
    The function \(w = w_0 + w_1\) now solves the equation we desired, with the desired bound (using that \(e^{-\delta S} > e^{-\frac{n+2}{2}S}\)).
\end{proof}
We now have a sufficient inverse for \(\calL\) that we can use to find minimal surfaces close to the catenoid with specified initial data \(g_\mathrm{II}\).

\subsubsection{Minimal surfaces close to \(C\) with initial data \(g_\mathrm{II}\)}

Recall that
\begin{equation}
    X_w = X_0 + w\phi^{\frac{2 - n}{2}}N_0
\end{equation}
gives a minimal surface with boundary data \(g_\mathrm{II} = \pi_\mathrm{II}(w)\) (denoted \(C(g_\mathrm{II})\)) when
\begin{equation}
    \calL w = \phi^{\frac{2-n}{2}} Q_{2}\left(\phi^{-\frac{n}{2}} w\right)+\phi^{\frac{n}{2}} Q_{3}\left(\phi^{-\frac{n}{2}} w\right).
\end{equation}
Recall also that we will want to match the boundary data of \(C(g_\mathrm{II})\) with the graphical surface \(\Sigma^c_{r_0, \veps}(h, \calA, h_\mathrm{II})\) that we produce later. We run into two issues regarding this goal. 

Firstly, \(\Sigma^c_{r_0, \veps}(h, \calA, h_\mathrm{II})\) is a modification of a graphical piece \(\Sigma^c_{r_0} = B_{r_0} \cap \Sigma\) of the original surface \(\Sigma\). It could be the case that \(r_0 \ll 1\), and then \(\Sigma^c_{r_0, \veps}(h, \calA, h_\mathrm{II})\) would be too small for the neck of \(C\). So, instead of gluing \(C\) we glue a catenoid scaled by a factor of \(\veps^{\frac{1}{n-1}}\) for some small \(\veps > 0\). We now make two crucial definitions
\begin{defn}\label{defn:r_veps and s_veps}
    For all \(\veps \in (0, 1)\), set
    \begin{equation}
        s_\veps = \frac{1}{(n-1)(3n-2)}\log \veps < 0, \quad r_\veps = \veps^{\frac{1}{n-1}}\phi(s_\veps)
    \end{equation}
    In particular, \(s_\veps\) represents a ``cut-off'' parameter, where we decide we wish to cut the half catenoidal neck. Then, \(r_\veps\) is the radius of the neck at the cut-off point. Using asymptotics of \(\phi\), we also have
    \begin{equation}
        r_\veps \sim \veps^{\frac{3}{3n-2}}.
    \end{equation}
    
\end{defn}
\begin{defn}
    The catenoid \(C_\veps\) will be \(C\) scaled down by \(\veps^\frac{1}{n-1}\), cut off at the parameter \(s_\veps\). That is \(C_\veps\) is the surface parametrized by, 
    \begin{equation}
        (s_\veps, \infty) \times S^{n-1} \ni (s, \theta) \mapsto \veps^\frac{1}{n-1}(\phi(s)\theta, \psi(s)).
    \end{equation}
\end{defn}

The second issue we have with gluing \(C_\veps\) to \(\Sigma^c_{r_0, \veps}(h, \calA, h_\mathrm{II})\) is the issue of parametrization. So far, we have been parametrizing surfaces close to \(C\) as normal graphs. But, \(\Sigma^c_{r_0, \veps}(h, \calA, h_\mathrm{II})\) will be a usual Cartesian coordinate graph of the form \((x, f(x))\). We thus modify slightly how we parametrize surfaces close to \(C_\veps\) to make the gluing easier. Specifically, let \(N_\veps(s)\) be a vector field on \(C_\veps\) which is equal to the vertical vector field \((0, \ldots, 0, 1)\) in a neighborhood of the neck (e.g. when \(s < s_\veps + 1\)), and equal to the normal field \(N_0\) away from this neighborhood of the neck (e.g. when \(s \ge s_\veps + 1\)). Thus, we are looking at hypersurfaces \(C_\veps(g_\mathrm{II})\) parametrized by
\begin{equation}
    X_w = \veps^{\frac{1}{n-1}}X_0 + w\phi^{\frac{2-n}{2}}N_\veps,
\end{equation}
and we want to find conditions on \(w\) which make the surface parameterized by \(X_w\) a minimal surface.

First, recall Proposition \ref{prop:3.1}. Using \(N_\veps\) we find that 
\begin{prop}\label{prop: our version of the modified catenoid min surface problem}
    The hypersurface parametrized by \(X_w = \veps^{\frac{1}{n-1}}X_0 + w\phi^{\frac{2-n}{2}}N_\veps\) with boundary data
    \begin{equation}
        g_\mathrm{II} = \phi^\frac{n-2}{2}(s_\veps)h_\mathrm{II}
    \end{equation}
    for \(h_\mathrm{II} \in \pi_\mathrm{II}(C^{2, \alpha}(\mathbb{S}^{n-1}))\) is minimal if and only if it solves

    \begin{equation}\label{eq: critical system for finding minimal surfaces near a catenoid}
        \begin{cases}
        \calL w=\bar{Q}_{\varepsilon}(w) & \text { in  \((s_\veps,\infty) \times S^{n-1}\)} \\ \pi_\mathrm{II}(w) = g_\mathrm{II} & \text {on \(\{s_\veps\} \times S^{n-1}\)}\end{cases} \quad \text{with} \quad w \in \mathcal{C}^{2, \alpha}\left([s_\veps, \infty) \times S^{n-1}\right).
    \end{equation}
    where
    \begin{equation}
        \begin{split}
        \bar{Q}_\veps(w)= & L_\veps w+\veps^{\frac{1}{n-1}} \phi^{\frac{2-n}{2}} \bar{Q}_{2, \veps}\left(\phi^{-\frac{n}{2}} \veps^{-\frac{1}{n-1}} w\right) \\
        & +\veps^{\frac{1}{n-1}} \phi^{\frac{n}{2}} \bar{Q}_{3, \veps}\left(\phi^{-\frac{n}{2}} \veps^{-\frac{1}{n-1}} w\right).
        \end{split}
    \end{equation}
    \(\bar{Q}_{2, \veps}\) and \(\bar{Q}_{3, \veps}\) are similar to \(Q_{2}\) and \(Q_{3}\) in Proposition \ref{prop:3.1}. In particular, there is \(c>0\) independent of \(\veps \in (0, 1)\) so that for all \(\veps \in (0,1)\) we have
    \begin{equation*}
        \left|\bar{Q}_{2, \varepsilon}(w)\right|_{0, \alpha\left([s, s+1] \times S^{n-1}\right)} \leq c|w|_{2, \alpha\left([s, s+1] \times S^{n-1}\right)}^{2}
    \end{equation*}
    for all \(w \in \mathcal{C}^{2, \alpha}\left([s, s+1] \times S^{n-1}\right)\). There are also \(c_0, c > 0\) independent of \(\veps\) such that
    \begin{equation*}
        \left|\bar{Q}_{3, \varepsilon}(w)\right|_{0, \alpha\left([s, s+1] \times S^{n-1}\right)} \leq c|w|_{2, \alpha\left([s, s+1] \times S^{n-1}\right)}^{3}
    \end{equation*}
    whenever \(|w|_{2, \alpha}\left([s, s+1] \times S^{n-1}\right) \leq c_{0}\).
    
    The operator \(L_\veps\) is linear and is the difference between the linearized mean curvature operator for hypersurfaces parameterized by \(\veps^\frac{1}{n-1}X_0 + w\phi^\frac{2 - n}{n}N_0\) and those parameterized by \(\veps^\frac{1}{n-1}X_0 + w\phi^\frac{n-2}{n}N_\veps\). The coefficients of \(L_\veps\) are supported in \(\left[s_\veps, s_\veps+2\right] \times S^{n-1}\) and are bounded by \(e^{(2n-2) s_\veps}\) in $\calC^{0, \alpha}\left(\left[s_\veps, s_\veps+2\right] \times S^{n-1}\right)\).
\end{prop}
\begin{proof}
    We have not provided the precise definition of \(N_\veps\), but as this and the rest of the important details are provided in \cite{Fakhi-Pacard}, we instead simply remark that the result follows from Proposition \ref{prop:3.1} and estimates such as
    \begin{equation}
        |\nabla^{k}\left(N_\veps \cdot N_{0}-1\right)| \leq c_{k} e^{(2 n-2) s_\veps}
    \end{equation}
    for \(k \ge 0\). In other words, the system we seek to solve here is a sufficiently controlled perturbation of the one in Proposition \ref{prop:3.1} that a solution to this system solves the original problem as well.
\end{proof}

Now our goal is to solve the system \eqref{eq: critical system for finding minimal surfaces near a catenoid}. First, fix \(\delta \in\left(-\frac{2+n}{2},-\frac{n}{2}\right)\) and \(\alpha \in(0,1)\). Using Proposition \ref{prop:4.6}, we may define 
\begin{equation*}
\tilde{w} = \mathcal{P}_{s_{\varepsilon}}\left(g_{\mathrm{II}}\right),
\end{equation*}
which solves
\begin{equation}
    \begin{cases}\mathcal{L} \tilde w=0 & \text { in \((s_\veps, \infty) \times \mathbb{S}^{n-1}\)} \\ \pi_\mathrm{II}(w)=g_{\mathrm{II}} & \text { on \(\{s_\veps\} \times \mathbb{S}^{n-1}\)}\end{cases}
\end{equation}
with the bound
\begin{equation}
\|\tilde{w}\|_{2, \alpha, \delta} \leq c e^{-\delta s_\veps}\left\|g_{\mathrm{II}}\right\|_{2, \alpha} .
\end{equation}
We now seek to find a solution of the form \(w=\tilde{w}+v\), where \(v \in \mathcal{C}_{\delta}^{2, \alpha}\left(\left[s_{\varepsilon},+\infty\right) \times S^{n-1}\right)\) is a function which solves
\begin{equation}
    \begin{cases}\mathcal{L} v=\bar{Q}_{\varepsilon}(\tilde{w}+v) & \text {in \(\left(s_{\varepsilon},\infty\right) \times \mathbb{S}^{n-1}\)} \\ \pi_{\mathrm{II}}(v)=0 & \text {on \(\left\{s_{\varepsilon}\right\} \times \mathbb{S}^{n-1}\)}.\end{cases}
\end{equation}
Now using Proposition \ref{prop:4.3}, let
\begin{equation}
    \calN_{\veps}(v) = \mathcal{G}_{s_\veps}\left(\bar{Q}_\veps(\tilde{w}+v)\right).
\end{equation}
If we can find a fixed point \(v_0\) of \(\calN_\veps\), then we will have solved \eqref{eq: critical system for finding minimal surfaces near a catenoid}.
\begin{prop}\label{prop:5.1}
     Fix \(\delta \in\left(-\frac{n+2}{2},-\frac{n}{2}\right)\) and \(\alpha \in(0,1)\). For all \(\kappa>0\) there are constants \(c_{\kappa}>0\) and \(\veps_0>0\) such that for all \(\veps \in\left(0, \varepsilon_{0}\right]\) and \(h_{\mathrm{II}} \in \pi_{\mathrm{II}}\left(\mathcal{C}^{2, \alpha}\left(\mathbb{S}^{n-1}\right)\right)\) with
    \begin{equation*}
        \left\|h_{\mathrm{II}}\right\|_{2, \alpha} \leq \kappa r_\veps^{2} 
    \end{equation*}
    the map \(\calN_\veps\) is a contraction mapping in the ball
    \begin{equation}
        B = \left\{v:\|v\|_{2, \alpha, \delta} \leq c_{\kappa} e^{\left(\frac{3 n-2}{2}-\delta\right) s_{\varepsilon}} r_{\varepsilon}^{2}\right\}.
    \end{equation}
    Thus, \(\calN_\veps\) has a unique fixed point in this ball.
\end{prop}
\begin{proof}
    It suffices to find a constant \(c_\kappa\) so that
    \begin{equation}
        \left\|\calN_\veps(0)\right\|_{2, \alpha, \delta} \leq \frac{c_{\kappa}}{2} e^{\left(\frac{3 n-2}{2}-\delta\right) s_\veps} r_\veps^{2} \quad \text{and} \quad \left\|\calN_\veps\left(v_{2}\right)-\calN_\veps\left(v_{1}\right)\right\|_{2, \alpha, \delta} \leq \frac{1}{2}\left\|v_{2}-v_{1}\right\|_{2, \alpha, \delta},
    \end{equation}
    for all \(v_{1} v_{2} \in B\). For the first inequality, we use Proposition \ref{prop:4.3} to find
    \begin{equation}
        \left\|\calN_\veps(0)\right\|_{2, \alpha, \delta} \le c\|\bar Q_\veps(\tilde w)\|_{0, \alpha, \delta}.
    \end{equation}
    Now using that \(\phi(s) \sim e^{-s}(1 + O(e^{(2n - 2)s})\) as \(s \rightarrow -\infty\) and recalling the definitions of \(s_\veps\) and \(r_\veps\), we get the bound
    \begin{equation}
        \|\tilde{w}\|_{2, \alpha, \delta} \leq c e^{-\delta s_\veps}\left\|g_{\mathrm{II}}\right\|_{2, \alpha} \le c\kappa e^{\left(\frac{2 - n}{2} - \delta\right)s_\veps}r_\veps^2,
    \end{equation}
    from which it follows
    \begin{equation}
        \left\|L_{\varepsilon} \tilde{w}\right\|_{0, \alpha, \delta} \le c \kappa e^{\left(\frac{3 n-2}{2}-\delta\right) s_{\varepsilon}} r_{\varepsilon}^{2},
    \end{equation}
    \begin{equation}
        \left\|\varepsilon^{\frac{1}{n-1}} \phi^{\frac{2-n}{2}} \bar{Q}_{2, \varepsilon}\left(\phi^{-\frac{n}{2}} \varepsilon^{-\frac{1}{n-1}} \tilde{w}\right)\right\|_{0, \alpha, \delta} \le c \kappa^{2} e^{(2 n-2-2 \delta) s_{\varepsilon}} r_{\varepsilon}^{2},
    \end{equation}
    and
    \begin{equation}
        \left\|\varepsilon^{\frac{1}{n-1}} \phi^{\frac{n}{2}} \bar{Q}_{3, \varepsilon}\left(\phi^{-\frac{n}{2}} \varepsilon^{-\frac{1}{n-1}} \tilde{w}\right)\right\|_{0, \alpha, \delta} \le c \kappa^{3} e^{\left(\frac{11 n-10}{2}-\delta\right) s_{\varepsilon}} r_{\varepsilon}^{2}\quad \text{when}\quad \left\|\varepsilon^{-\frac{1}{n-1}} \phi^{-\frac{n}{2}} \tilde{w}\right\|_{2, \alpha, 0} \leq c_{0}.
    \end{equation}
    The condition on the right can be achieved by choosing a sufficiently small \(\veps_0\) and taking \(\veps < \veps_0\). Putting these estimates together, we get
    \begin{equation}
        \begin{split}
            \left\|\calN_\veps(0)\right\|_{2, \alpha, \delta} & \le c\|\bar Q_\veps(\tilde w)\|_{0, \alpha, \delta}
            \\ & \le \left\|L_{\varepsilon} \tilde{w}\right\|_{0, \alpha, \delta} + \left\|\varepsilon^{\frac{1}{n-1}} \phi^{\frac{2-n}{2}} \bar{Q}_{2, \varepsilon}\left(\phi^{-\frac{n}{2}} \varepsilon^{-\frac{1}{n-1}} \tilde{w}\right)\right\| 
            \\ & \quad + \left\|\varepsilon^{\frac{1}{n-1}} \phi^{\frac{n}{2}} \bar{Q}_{3, \varepsilon}\left(\phi^{-\frac{n}{2}} \varepsilon^{-\frac{1}{n-1}} \tilde{w}\right)\right\|_{0, \alpha, \delta} 
            \\ & \le \tilde{c} \kappa e^{\left(\frac{3 n-2}{2}-\delta\right) s_{\varepsilon}} r_{\varepsilon}^{2},
        \end{split}
    \end{equation}
    and so choosing \(c_\kappa \ge 2\tilde c \kappa\) gives us the desired estimate.

    To prove \(\left\|\calN_\veps\left(v_{2}\right)-\calN_\veps\left(v_{1}\right)\right\|_{2, \alpha, \delta} \leq \frac{1}{2}\left\|v_{2}-v_{1}\right\|_{2, \alpha, \delta}\), we use a very similar method, recalling the well known facts
    \begin{equation}
        |v_1 - v_2|^2 = |v_1 - v_2||v_1 + v_2|, \quad |v_1 - v_2|^3 = |v_1 - v_2||v_1^2 + v_2v_2 + v_2^2|
    \end{equation}
    in order to get bounds in terms of \(\|v_1 - v_2\|_{2, \alpha, \delta}\).
\end{proof}

\subsubsection{Analysis of Cauchy Data}
It now follows that given a function \(h_\mathrm{II} \in \pi_\mathrm{II}(C^{2, \alpha}(\bbS^{n-1}))\), we can find a minimal surface
\begin{equation}
    X_w = \veps^\frac{1}{n-1}X_0 + w\phi^\frac{2-n}{2}N_\veps
\end{equation}
with boundary data \(\pi_\mathrm{II}(w) = \phi^\frac{n-2}{2}(s_\veps)h_\mathrm{II}\). We denote this minimal surface by \(C_\veps(h_\mathrm{II})\), and we note that by the definitions of \(g_\mathrm{II}\) and \(N_\veps\), the boundary of the perturbed catenoid \(\partial C_\veps(h_\mathrm{II})\) is the graph of \(h_\mathrm{II}\) over a sphere \(r_\veps\bbS^{n-1}\).

Importantly, we note that because \(w \in \calC_\delta^{2,\alpha}\) for some \(\delta\), the end of \(C_\veps(h_\mathrm{II})\) is uniformly asymptotic to the end of \(C_\veps\).

Our ultimate goal is to glue \(C_\veps(h_\mathrm{II})\) to \(\Sigma^c_{r_0, \veps}(h_\mathrm{I}, \calA, h_\mathrm{II})\). To go about this, we analyze the Cauchy data of the surface \(C_\veps(h_\mathrm{II})\), with the intent to match it to the data of \(\Sigma^c_{r_0, \veps}(h_\mathrm{I}, \calA, h_\mathrm{II})\) later.

First, we note that \(C_\veps(h_\mathrm{II})\) will be a graph over \(x_{n+1} = 0\) in an annulus around the origin. Specifically, we may say that \(C_\veps(h_\mathrm{II})\) is the graph of a function \(U_{\veps, h_\mathrm{II}}(x)\) over an annulus \(B_{r_\veps}\backslash B_{r_\veps/2}\). We have
\begin{defn}
    The \textit{catenoid Cauchy data} of \(C_\veps(h_\mathrm{II})\) is a map \(\calS_\veps : \pi_\mathrm{II}(C^{2, \alpha}(\bbS^{n-1})) \rightarrow C^{2, \alpha}(\bbS^{n-1}) \times C^{1, \alpha}(\bbS^{n-1}))\), given by
    \begin{equation}
        \calS_\veps(h_\mathrm{II})(\theta) = (U_{\veps, h_\mathrm{II}}(r_\veps, \theta), r_\veps\partial_rU_{\veps, h_\mathrm{II}}(r_\veps, \theta)).
    \end{equation}
    In other words, \(\calS_\veps(h_\mathrm{II})\) is just the Neumann boundary data for the PDE we solved to produce \(C_\veps(h_\mathrm{II})\). In terms of the parametrizing functions \(\phi\) and \(\psi\), we can further write
    \begin{equation}
        \calS_\veps(h_{\mathrm{II}})=\left(\phi^{\frac{2-n}{2}}(s_\veps) w(s_{\varepsilon}, \cdot), \frac{\phi(s_\veps)}{\phi'(s_\veps)}\left(\varepsilon^{\frac{1}{n-1}} \psi'(s_\veps)+\partial_{s}\left(\phi^{\frac{2-n}{2}} w\right)\left(s_\veps, \cdot\right)\right)\right).
    \end{equation}
\end{defn}
In order to compare the cauchy data of \(C_\veps(h_\mathrm{II})\) and \(\Sigma^c_{r_0, \veps}(h_\mathrm{I}, \calA, h_\mathrm{II})\), we will take an intermediate step of comparing the Cauchy data of each surface to the Cauchy data of a similar, but simpler problem. In particular,
\begin{defn}
    The \textit{simple catenoid Cauchy data} of \(C_\veps(h_\mathrm{II})\) is a map \(\calS_0 : \pi_\mathrm{II}(C^{2, \alpha}(\bbS^{n-1})) \rightarrow C^{2, \alpha}(\bbS^{n-1}) \times C^{1, \alpha}(\bbS^{n-1}))\) given by
    \begin{equation}
        \calS_0(h_\mathrm{II})(\theta) = \left(h_{\mathrm{II}},-\veps r_\veps^{2-n}-\frac{n-2}{2} h_{\mathrm{II}}+D_{\theta} h_{\mathrm{II}}\right)
    \end{equation}
    This is essentially the Cauchy data of the problem \(\Delta_0 w = 0\) in \((s_\veps, \infty) \times \bbS^{n-1}\) and \(w = h_\mathrm{II}\) on \(\{s_\veps\} \times \bbS^{n-1}\). Since \(\Delta_0\) is just a Laplacian plus a constant term, this Cauchy data map is much simpler and indeed gives an isomorphism between certain spaces (which will be discussed later).
\end{defn}

We now conclude our discussion of perturbing the catenoid by comparing the catenoid Cauchy data and the simple catenoid Cauchy data.

\begin{prop}\label{prop:5.2}
    The maps \(\calS_\veps\) and \(\calS_0\) are continuous. There also exists a constant \(c > 0\) so that the following holds: given \(\kappa>0\), there is a \(\veps_0 > 0\) so that for all \(\veps \in (0, \veps_0]\) and \(\|h_{\mathrm{II}}\|_{2, \alpha} \le \kappa r_\veps^{2}\), we have

\begin{equation}
    \left\|\left(\mathcal{S}_{\varepsilon}-\mathcal{S}_{0}\right)\left(h_{\mathrm{II}}\right)\right\|_{\mathcal{C}^{2, \alpha} \times \mathcal{C}^{1, \alpha}} \leq c r_{\varepsilon}^{2}.
\end{equation}
\end{prop}
\begin{proof}
    This follows directly from the asymptotics on \(\phi\) and \(\psi\), and the bound on 
    \begin{equation}
        \|w\|_{2, \alpha, \delta} \le \|\tilde w\|_{2, \alpha,\delta} + \|v_0\|_{2, \alpha, \delta}
    \end{equation}
    (for some \(\delta\)) that follows from the bound on \(\|h_\mathrm{II}\|_{2, \alpha}\) and the arguments in Propositions \ref{prop: our version of the modified catenoid min surface problem} and \ref{prop:5.1}.
\end{proof}
\subsection{Perturbing the intermediate compact piece to produce \(\Sigma^c_{r_0, \veps}(h_\mathrm{I}, \calA, h_\mathrm{II})\)}

We now describe how to produce the surface  \(\Sigma^c_{r_0, \veps}(h_\mathrm{I}, \calA, h_\mathrm{II})\) that we have been mentioning. Recalling the basic setup: we start with a surface \(\Sigma\), and cut out a graphical piece \(\Sigma^c_{r_0} = B_{r_0} \cap \Sigma\), with the remaining noncompact piece denoted \(\Sigma_{r_0}\). We then wish to perturb \(\Sigma^c_{r_0}\) to glue a catenoid \(C_\veps\).

To do this, we first produce a surface \(\Sigma^c_{r_0, \veps}\) as follows; if \(\Sigma^c_{r_0}\) is a graph \((x, u(x))\) over \(B_{r_0} \cap \{x_{n+1} = 0\}\), then \(\Sigma^c_{r_0, \veps}\) is a graph
\begin{equation}
    \left(x, u(x) + \frac{\veps}{n - 2}\gamma_0(x)\right), \quad x \in B_{r_0} \backslash B_{r_\veps/2},
\end{equation}
where \(\gamma_0\) is the Green's function of \(\Sigma^c_{r_0}\) at the origin.

The surface \(\Sigma^c_{r_0, \veps}\) now has two boundaries, and we will solve perturbation problems on both boundaries in order to match \(\Sigma^c_{r_0, \veps}\) with \(C_\veps\) and \(\Sigma_{r_0}\). When perturbing \(C_\veps\), we encountered issues with Jacobi fields that limit the degrees of freedom we had in specifying boundary data. A similar issue will arise with perturbing \(\Sigma_{r_0}\). In order to regain these degrees of freedom, we make some basic geometric transformations to \(\Sigma^c_{r_0, \veps}\) (translations, rotations, and edits to the coefficient \(\veps/(n-2)\) on the Green's function), which will be represented by a term called \(\calA\). The resultant surface perturbed on both boundaries and modified by \(\calA\) will be called \(\Sigma^c_{r_0, \veps}(h_\mathrm{I}, \calA, h_\mathrm{II})\). We will later seek to match the boundary data with a perturbed noncompact piece \(\Sigma_{r_0}(h_\mathrm{I})\) and the perturbed catenoid \(C_\veps(h_\mathrm{II})\).

\subsubsection{Definitions}

We start by perturbing the boundary of \(\Sigma^c_{r_0, \veps}\) common to the catenoid \(C_\veps\). We proceed analogously to our perturbation of the Catenoid: 
\begin{enumerate}
    \item we define a relevant function space
    \item we define the linearized mean curvature operator and the full mean curvature operator
    \item we discuss how to solve the linearized mean curvature operator for certain key cases
    \item we use a contraction mapping/fixed point argument to solve the full mean curvature operator
\end{enumerate}
\begin{defn}\label{defn:6.1}
For all regular open subsets $\Omega \subset \mathbb{R}^{n}$ with $0 \in \Omega$, for all $k \in \mathbb{N}$, $\alpha \in(0,1)$ and $v \in \mathbb{R}$, the space $\mathcal{C}_{v}^{k, \alpha}(\bar{\Omega} \backslash\{0\})$ is defined to be the space of functions $w \in \mathcal{C}_{\text {loc }}^{k, \alpha}(\bar{\Omega} \backslash\{0\})$ for which the following norm is finite
$$
\|w\|_{k, \alpha, v} \equiv|w|_{k, \alpha, \bar{\Omega} \backslash B_{r_{0}}}+\sup _{0<2 r \leq r_{0}} r^{-\nu}[w]_{k, \alpha,[2 r, r]},
$$
where, by definition
$$
[w]_{k, \alpha,[2 r, r]} \equiv \sum_{j=0}^{k} r^{j} \sup _{r \leq|x| \leq 2 r}\left|\nabla^{j} w\right|+r^{k+\alpha} \sup _{r \leq\left|x_{i}\right| \leq 2 r, x_{i} \neq x_{j}} \frac{\left|\nabla^{k} w\left(x_{1}\right)-\nabla^{k} w\left(x_{2}\right)\right|}{\left|x_{1}-x_{2}\right|^{\alpha}}
$$
and where $r_{0}>0$ is fixed in such a way that $B_{r_{0}} \subset \bar{\Omega}$.
\end{defn}
\begin{defn}\label{defn:6.2}
For all $\bar{r}<r_{0}$, the space $\mathcal{C}_{v}^{k, \alpha}\left(\bar{\Omega} \backslash B_{\bar{r}}\right)$ is defined to be the space of restrictions to $\bar{\Omega} \backslash B_{\bar{r}}$ of functions $w \in \mathcal{C}_{\nu}^{k, \alpha}(\bar{\Omega} \backslash\{0\})$, endowed with the induced norm. 
\end{defn}
\subsubsection{Notation}
Let $\Sigma_{0}$ be a hypersurface given as a graph
\begin{equation}
\bar{\Omega} \ni x \longrightarrow(x, u(x)) \in \Sigma_{0} \subset \mathbb{R}^{n+1} .
\end{equation}
Then linearized mean curvature operator about $\Sigma_{0}$ is given explicitly by
\begin{equation}\label{eqn:6.3}
\Lambda_{u} w \equiv \operatorname{div}\left(\frac{\nabla w}{\left(1+|\nabla u|^{2}\right)^{1 / 2}}-\frac{\nabla u \cdot \nabla w}{\left(1+|\nabla u|^{2}\right)^{3 / 2}} \nabla u\right).
\end{equation}
for any  $w\in \calC^1$ perturbation of $\Sigma_0$ given by $x\to (x,u(x)+w(x))\in \R^{n+1}.$
\begin{defn}\label{defn:sigma_eps}
For all $\varepsilon>0$, we can define $\Sigma_{\varepsilon}$ to be the hypersurface parameterized by
$$
\Omega \backslash\{0\} \ni x \longrightarrow\left(x, u(x)+\frac{\varepsilon}{n-2} \gamma_{0}(x)\right) \in \mathbb{R}^{n+1} \quad \text { if } \quad n \geq 4
$$
and by
$$
\Omega \backslash\{0\} \ni x \longrightarrow\left(x, u(x)+\frac{\varepsilon}{n-2}\left(\gamma_{0}(x)-a_{0}\right)\right) \in \mathbb{R}^{4} \quad \text { if } \quad n=3 .
$$
\end{defn}

\subsubsection{Assumptions}
We will make the following technical assumptions, which will ensure that all the results will hold uniformly in $\alpha, u$ and $\Omega$ and will only depend on the constants $r_{0}, \eta_{0}$ and $\eta_{\nu}$ which are defined below. The importance of these assumptions will become clear within the subsequent sections.
\begin{itemize}
    \item[(A.1)] $B_{r_{0} / 2} \subset \Omega \subset B_{2 r_{0}}$.
    \item[(A.2)] $u(0)=0$ and $|\nabla u(0)|\leq r_\veps$. Stated differently, 0 belongs to $\Sigma_{0}$ and the tangent space at $0$ is always the hyperplane that is close to the plane $x_{n+1}=0$ by an $r_{\veps}>0$ amount. This is a key assumption that differentiates our construction from the one in Fakhi-Pacard and introduces technical difficulties.
    \item[(A.3)]  $\|u\|_{\mathcal{C}^{2, \alpha}\left(\overline{B_{2 r_{0}}}\right)} \leq \eta_{0}$ and $\|u\|_{\mathcal{C}^{3, \alpha}\left(\overline{B_{r_{0}}}\right)} \leq \eta_{0}$.
    \item[(A.4)] The operator $\Lambda_{u}$ defined from $\left[\mathcal{C}^{2, \alpha}(\bar{\Omega})\right]_{\mathcal{D}}$ into $\mathcal{C}^{0, \alpha}(\bar{\Omega})$ is an isomorphism where by definition
    \begin{equation}
    \left[\mathcal{C}^{2, \alpha}(\bar{\Omega})\right]_{\mathcal{D}} \equiv\left\{w \in \mathcal{C}^{2, \alpha}(\bar{\Omega}): w=0 \quad \text { on } \quad \partial \Omega\right\} .
    \end{equation}
Moreover $\left\|\Lambda_{u}^{-1}\right\|_{\left(\mathcal{C}^{0, \alpha}, \mathcal{C}^{2, \alpha}\right)} \leq \eta_{0}$ where $\eta_{0}$.
    \item[(A.5)] Assume that $v \in(-n, 1-n)$ is fixed. For all $r<r_{0}$, there exists an operator $\Gamma_{u, r}$ defined from $\mathcal{C}_{v-2}^{0, \alpha}\left(\bar{\Omega} \backslash B_{r}\right)$ into $\left[\mathcal{C}_{v}^{2, \alpha}\left(\bar{\Omega} \backslash B_{r}\right)\right]_{\mathcal{D}, n}$, such that $\Lambda_{u} \circ \Gamma_{u, r}=I d$. Here by definition
    \EQ{
    \begin{aligned}
{\left[\mathcal{C}_{v}^{2, \alpha}\left(\bar{\Omega} \backslash B_{r}\right)\right]_{\mathcal{D}, n} \equiv } & \left\{w \in \mathcal{C}_{v}^{2, \alpha}\left(\bar{\Omega} \backslash B_{r}\right): \quad w=0 \quad \text { on } \quad \partial \Omega,\right. \\
& \text { and } \left.\quad \pi_{\mathrm{II}}(w)=0 \quad \text { on } \quad \partial B_{r}\right\} .
\end{aligned}
    }
Moreover $\left\|\Gamma_{u, r}\right\|_{\left(\mathcal{C}_{\nu-2}^{0, \alpha}, \mathcal{C}_{\nu}^{2, \alpha}\right)} \leq \eta_{\nu}$, where $\eta_{\nu}$ does not depend on $r<r_{0}$.
\end{itemize}
Though this will never be explicit in the statements of the results, all the bounds we will obtain in the following sections will not depend on $u$ or $\Omega$ satisfying the assumptions above but will only depend on $r_{0}, r_{\veps}, \eta_{0}$ and $\eta_{\nu}$.

\subsubsection{The full mean curvature operator}
Let \((x, u(x))\) be a minimal graph. We have
\begin{equation}
    H_u = \div\left(\frac{\nabla u}{(1 + |\nabla u|^2)^{1/2}}\right). 
\end{equation}
We would like to perturb \(\Sigma\) vertically. I.e. we want to consider graphs \((x, u(x) + w(x))\). We have
\begin{lem}\label{lem:6.1}
 Assume (A.1), (A.2), and (A.3) hold. The linearized mean curvature operator $\Lambda_{u}$ can be expanded as
\begin{equation}\label{eqn:6.4}
\Lambda_{u}=\operatorname{div}\left(\nabla+\Lambda_{u}^{\prime}\right), \tag{6.4}
\end{equation}
and where $\Lambda_{u}^{\prime}$ is a first order partial differential operator without any zero order terms and all of whose coefficients are bounded functions in $\mathcal{C}_{2}^{1, \alpha}(\bar{\Omega} \backslash\{0\}) \cap \mathcal{C}_{2}^{2, \alpha}\left(\overline{B_{r_{0} / 2}} \backslash\{0\}\right)$.    
\end{lem}
\begin{proof}
This follows directly from (6.3).    
\end{proof}

\begin{lem}\label{lem:6.2}
For a surface parametrized by \((x, u(x) + w(x))\), the mean curvature of the surface is given by
    \begin{equation}\label{eqn:6.5}
        H_{u + w} = H_u + \Lambda_uw - \div(r_\veps Q_u'(\nabla w) + Q_u''(\nabla w)),
    \end{equation}
    where \(Q_u'\) is homogeneous of degree 2 with coefficients bounded uniformly and independently of \(r_\veps\) in a neighborhood of \(0\). The function \(Q_u''\) consists of nonlinear terms of order 3 and higher. Furthermore, all partial derivatives of \(Q_u''\) are bounded uniformly in a neighborhood of 0.
\end{lem}
\begin{proof}
    To understand the perturbed operator \(H_{u + w}\), we will Taylor expand \(f(s) = H_{u + sw}\) in \(s\):
    \begin{equation}
        f(s) = \div\left(\frac{\nabla u + s\nabla w}{(1 + |\nabla u + s\nabla w|^2)^{1/2}}\right), \quad f(0) = H_u
    \end{equation}
    Then we have
    \begin{equation}
        \begin{split}
            f'(s) & = \div\left(\frac{\nabla w}{(1 + |\nabla u + s\nabla w|^2)^{1/2}} - \frac{(\nabla u + s\nabla w)(\nabla u \cdot \nabla w + s|\nabla w|^2)}{(1 + |\nabla u + s\nabla w|^2)^{-1/2}}\right),
            \\
            \\ f'(0) & = \div\left(\frac{\nabla w}{(1 + |\nabla u|^2)^{1/2}} - \frac{\nabla u \cdot \nabla w}{(1 + |\nabla u|^2)^{-1/2}}\nabla u\right) = \Lambda_uw.
        \end{split}
    \end{equation}
    Finally we compute
    \begin{equation}
        f''(0) = \div\left(\frac{|\nabla w|^2\nabla u + (\nabla u \cdot \nabla w)\nabla w - (\nabla u \cdot \nabla w)\nabla w}{(1 + |\nabla u|^2)^{3/2}} + \frac{3(\nabla u \cdot \nabla w)^2\nabla u}{(1 + |\nabla u|^2)^{5/2}}\right).
    \end{equation}
    We remark that at each step, \(f^{(n)}(0)\) will be an expression of order at least \(n\) in \(\nabla w\). We will need to use bounds \(|\nabla w| \le r^{\alpha}\) in the future in order to bound \(|H_{u + w}|\). In doing this, we will need to use that \(|\nabla u| \le r_{\veps}\) to get the right power of control on \(|H_{u + w}|\). For simplicity and clarity, we will thus write
    \begin{equation}
        f''(0) = \div (-r_{\veps}Q_u'(\nabla w)),
    \end{equation}
    where \(Q_u'\) is homogeneous and quadratic in its input, with coefficients bounded uniformly in \(C^{2, \alpha}(\Omega) \cup C^{1, \alpha}(\bar \Omega)\), and independently of \(r_{\veps}\) (which we know because \(|\nabla u|/r_\veps < 1\). All in all, we have
    \begin{equation}
        H_{u + sw} = f(0) + sf'(0) + \frac{s^2f''(0)}{2} + O(|\nabla w|^3).
    \end{equation}
    Taking \(s = 1\) finishes the proof.
\end{proof}

\subsection{Special perturbations of \(\Sigma\)}

We will be making two kinds of modifications to our minimal graphs to facilitate gluing on a half catenoid. First we will send \(u(x)\) to \(u(x) +  \frac{\veps}{n-2}\gamma_0\) where \(\gamma_0\) is a Green's function on \(\Sigma\) centered around 0. We will name the resultant surface \(\Sigma_\veps\). In order to generate more dimensions of freedom to execute a fixed point argument later, we will make some small rigid motion modifications to \(\Sigma_{\veps}\). We will allow ourselves to make small translations, small rotations, and also small changes to the factor \(\veps\). We collect the modifications in a quantity \(\mathcal{A}\) and we will show that the resultant surface \(\Sigma_{\veps, \mathcal{A}}\) can be parametrized as \(u(x) + w_{\veps, \mathcal{A}}(x)\), where \(|\nabla^k w_{\veps, \mathcal{A}}| \le cr^{-k}(r_\veps r + \veps r^{2 - n})\).

Denote $\gamma_0$ as the Green's function (which exists by (A.4)) for the linearized operator $\Lambda_u$, defined in \eqref{eqn:6.4}. Then by definition,
\begin{align}\label{eqn:6.6}
\Lambda_{u} \gamma_{0}=-(n-2)\left|\bbS^{n-1}\right| \delta_{0}, \quad \text { in } \quad \Omega, 
\end{align}
with $\gamma_{0}=0$ on $\partial \Omega$, where $\left|S^{n-1}\right|$ is the volume of the unit sphere. 
\begin{lem}\label{lem:6.2}
Assume that {(A.1)}--{(A.4)} hold and that $\gamma_{0}$ is the solution of \eqref{eqn:6.6}. Then, there exists $c>0$ such that, for all $k \leq 3$,
\begin{align}
\left|\nabla^{k}\left(\gamma_{0}-r^{2-n}\right)\right| & \leq c r^{4-n-k}, & & \text { if } n \geq 5, \\
\left|\nabla^{k}\left(\gamma_{0}-r^{-2}\right)\right| & \leq c r^{-k} \log 1 / r, & & \text { if } n=4, \\
\left|\nabla^{k}\left(\gamma_{0}-r^{-1}-a_{0}\right)\right| & \leq c r^{1-k} \log 1 / r, & & \text { if } n=3,   
\end{align}
in $B_{r_{0}}\setminus B_{r_{\veps/4}}$, for some constant $a_{0} \in \mathbb{R}$.    
\end{lem}
\begin{proof}
Since for any $x\in B_{r_{0}}\setminus B_{r_{\veps/4}}$, we have that $r=|x|\gec r_{\veps}$, the operator $|\Lam_u'|\lec r$ and therefore we obtain the same estimates as Lemma 6.3 in \cite{Fakhi-Pacard}.
\end{proof}
We next compare the mean curvature of the hypersurface $\Sigma_{\varepsilon}$ with the mean curvature of the initial hypersurface $\Sigma_{0}$.
\begin{lem}\label{prop:6.1}
Assume that {(A.1)}--{(A.4)} hold. The derivatives of $H_{\varepsilon}$, the mean curvature of $\Sigma_{\varepsilon}$, can be estimated by
\EQ{
\left|\nabla^{k}\left(H_{\varepsilon}-H_{0}\right)\right| \leq c\left(r^{-k}\left(\varepsilon^{2} r^{2-2 n}+\varepsilon^{3} r^{2-3 n}\right)\right), \quad \text { for all } \quad \varepsilon^{\frac{3}{3 n-2}} \leq r,
}
where $H_{0}$ is the mean curvature of $\Sigma_{0}$ and where $k=0,1$ and where $c>0$ does not depend on $\varepsilon \in\left(0, \varepsilon_{0}\right]$.    
\end{lem}
\begin{proof}
The result follows at once from \eqref{eqn:6.5}, with $w=\varepsilon \gamma_{0}$ and the fact that $r_{\veps}\simeq \veps^{3/(3n-2)}$, which allows us to replace $r_{\veps}$ by $r$ when estimating the mean curvature in \eqref{eqn:6.5}. Therefore, we can directly use the estimates from Lemma \ref{lem:6.2} where the terms involving $r_{\veps}$ are upper bounded by $r$, thus giving the same estimate as Proposition 6.1 in \cite{Fakhi-Pacard}. 
\end{proof}
\subsection{Geometric transformation of $\Sigma_{\veps}$}
We now perform some geometric transformations of the surface $\Sigma_{\varepsilon}$ by applying some rigid motion and also by modifying the parameter $\varepsilon$. In particular, we consider the following transformations:
\begin{itemize}
    \item $\veps\to e+\veps$ for a parameter $e\in (-\veps,\veps).$
    \item $\mathbb{R}^{n+1} \ni\left(x, x_{n+1}\right)\to (x+T,x_{n+1}+d)$, where $(T,d)\in \R^{n}\times \R$ represent translation in $\R^{n+1}$. 
    \item Given $R \in \mathbb{R}^{n}, R \neq 0$, we define the rigid motion corresponding to a rotation of angle $|R|$ in the plane spanned by the vectors $(0,1)\in \R^{n}\times \R$) and $(R/|R|, 0)$. For $R \neq 0$, this transformation can be described analytically by 
    \EQ{
    \mathbb{R}^{n+1} & \ni\left(x, x_{n+1}\right) \\
    & \rightarrow\left(x^{\perp}, 0\right)+\cos |R|\left(x^{\|}, x_{n+1}\right)-\frac{\sin |R|}{|R|}\left(x_{n+1} R,-R \cdot x^{\|}\right) \in \mathbb{R}^{n+1},
    }
where $x^{\|} \equiv \frac{x \cdot R}{|R|^{2}} R$ and $x^{\perp} \equiv x-x^{\|}$. 
\end{itemize}
Denote $\mathcal{A}=(T, R, d, e) \in \mathbb{R}^{n} \times \mathbb{R}^{n} \times \mathbb{R} \times \mathbb{R}$ as the set of parameters and by $\Sigma_{\varepsilon, \mathcal{A}}$ the resulting hypersurface. We define, the norm of $\mathcal{A}$ is given by
\EQ{\label{defn:calA-norm}
\|\mathcal{A}\| \equiv \varepsilon r_{\varepsilon}^{1-n}|T|_{\mathbb{R}^{n}}+r_{\varepsilon}|R|_{\mathbb{R}^{n}}+|d|+r_{\varepsilon}^{2-n}|e| .
}
We now compare $\Sigma_{\varepsilon, \mathcal{A}}$ with $\Sigma_{0}$.
\begin{prop}\label{prop:6.2}
Assume that {(A.1)}--{(A.4)} hold. Let $\kappa>0$ be given. There exists $c_{\kappa}>0$ and $\varepsilon_{0}>0$ such that, for all $\varepsilon \in\left(0, \varepsilon_{0}\right]$, if
\EQ{
\|\mathcal{A}\| \leq \kappa r_{\varepsilon}^{2}
}
then, the hypersurface $\Sigma_{\varepsilon, \mathcal{A}}$ can be locally parameterized as a vertical graph over the initial hypersurface $\Sigma_{0}$
\EQ{\label{eqn:6.7}
\overline{B_{r_{0} / 2}} \backslash B_{r_{\varepsilon} / 2} \ni x \longrightarrow\left(x, u(x)+w_{\varepsilon, \mathcal{A}}(x)\right) \in \Sigma_{\varepsilon, \mathcal{A}}, 
}
where the function $w_{\varepsilon, \mathcal{A}}$ satisfies $\left|\nabla^{k} w_{\varepsilon, \mathcal{A}}(x)\right| \leq c_{\kappa}\left(r^{-k}\left(r_{\varepsilon} r+\varepsilon r^{2-n}\right)\right)$, for all $k \leq 3$.
\end{prop} 
\begin{proof}
This follows from the argument in Proposition 6.3 in \cite{Fakhi-Pacard} since the estimates for the Green's function $\gamma_0$ are the same as in Lemma 6.3 in \cite{Fakhi-Pacard}.
\end{proof}
\begin{prop}\label{prop:6.3}
Assume that {(A.1)}--{(A.4)} hold. There exists $c>0$ and, for all $\kappa>0$ be given, there exists $\varepsilon_{0}>0$ (depending on $\kappa$) such that, for all $\varepsilon \in\left(0, \varepsilon_{0}\right]$ and for all $r \in\left[r_{\varepsilon} / 2,2 r_{\varepsilon}\right]$, the parameterization of $\Sigma_{\varepsilon, \mathcal{A}}$ has the following expansion
\EQ{\label{eqn:6.10}
x \longrightarrow\left(x, e \frac{r^{2-n}}{n-2}+\left(e \frac{r^{2-n}}{n-2}+d+R \cdot x+\varepsilon r^{-n} T \cdot x\right)+\bar{w}_{\varepsilon, \mathcal{A}}(x)\right), 
}
where, for all $0\leq k \leq 3$, the function $\bar{w}_{\varepsilon, \mathcal{A}}$ satisfies $\left|\nabla^{k} \bar{w}_{\varepsilon, \mathcal{A}}(x)\right| \leq c r_{\varepsilon}^{2-k}$ and the constant $c>0$ does not depend on $\kappa$ provided $\varepsilon$ is chosen small enough.    
\end{prop}
\begin{proof}
See, proof of Proposition 6.3 in \cite{Fakhi-Pacard}. Again, the key point in the argument is the sharp estimates on the Green's function $\ga_0$, which satisfy the same estimates as in Lemma 6.3 in \cite{Fakhi-Pacard} on the region $r\gec r_\veps$.
\end{proof}

\subsection{Invertibility of the linearized mean curvature operator around $\Sigma_{\veps,\calA}$}
The goal of this section is to study the linearized mean curvature operator around $\Sigma_{\veps,\calA}$ and its mapping properties. First, note that we can write $\Sigma_{\veps,\calA}$ as a hypersurface
\EQ{\label{eqn:7.1}
\bar{\Omega}_{\mathcal{A}} \backslash B_{r_{\varepsilon}} \ni x \longrightarrow\left(x, u(x)+w_{\varepsilon, \mathcal{A}}(x)\right),
}
where $\Omega_{\calA}$ denotes the projection onto the hyperplane $x_{n+1}=0$ of the image of $\Sigma_{0}$ and the function $w_{\varepsilon, \mathcal{A}}$ is the one defined in Proposition \ref{prop:6.2}. Thus, $\Sigma_{\varepsilon, \mathcal{A}}$ is the singular surface constructed in the previous section, which has been truncated. The linearized mean curvature operator about $\Sigma_{\varepsilon, \mathcal{A}}$ now has the following form
\EQ{\label{eqn:7.1}
\Lambda_{\varepsilon, \mathcal{A}}=\Lambda_{u}+\operatorname{div} \Lambda_{\veps, \calA}^{\prime}, \text{ on } \Omega_{\calA}\setminus B_{r_{\veps}},
}
where $\Lambda_{\varepsilon, \calA}^{\prime}$ is a first order partial differential operator. This follows from \eqref{eqn:6.5} by computing 
\EQ{
H_{u + w_{\veps,\calA}+s\phi} = H_u + \Lam_u w_{\veps,\calA} +  s \Lambda_u \phi - \div(r_\veps Q_u'(w_{\veps,\calA}+s\phi) + Q_u''(w_{\veps,\calA}+s\phi)),
}
for any test function $\phi\in C^\infty_c(\Omega_{\calA}\setminus B_{r_\veps})$ and taking the derivative as $s=0$. It is also clear then that the coefficients of $\Lam'_{\veps,\calA}$ can be bounded by $|\nabla w_{\veps,\calA}|$ and $|\nabla w_{\veps,A}|^2$ which arise from $Q'$ and $Q''$ respectively. Thus using Proposition \ref{prop:6.2}, to bound the coefficients as well as their $k$-derivatives for $k=0,1,2$ up to a constant (that depends on $\kappa>0$)  by $r^{-k}\left(\varepsilon r^{2-n}+r_{\varepsilon} r+\varepsilon^{2} r^{2-2 n}\right)$, provided $\|\mathcal{A}\| \leq \kappa r_{\varepsilon}^{2}$. Regarding the invertibility of this operator, we have the following result:
\begin{prop}\label{prop:7.1}
Assume that (A.1)--(A.5) hold. Fix $v \in(-n, 1-n), \alpha \in(0,1)$. Then, for all $\kappa>0$, there exists $\varepsilon_{0}>0$ and all $\varepsilon \in\left(0, \varepsilon_{0}\right]$, there exists an operator
\EQ{
\Gamma_{\varepsilon, \mathcal{A}}: \mathcal{C}_{\nu-2}^{0, \alpha}\left(\overline{\Omega_{\mathcal{A}}} \backslash B_{r_{\varepsilon}}\right) \longrightarrow \mathcal{C}_{\nu}^{2, \alpha}\left(\overline{\Omega_{\mathcal{A}}} \backslash B_{r_{\varepsilon}}\right),
}
such that, for all $f \in \mathcal{C}_{\nu-2}^{0, \alpha}\left(\overline{\Omega_{\mathcal{A}}} \backslash B_{r_{\varepsilon}}\right)$, the function $w=\Gamma_{\varepsilon, \mathcal{A}}(f)$ is a solution of the problem
\EQ{
\left\{\begin{array}{lll}
\Lambda_{\varepsilon, \mathcal{A}} w=f & \text { in } & \Omega_{\mathcal{A}} \backslash B_{r_{\varepsilon}} \\
\pi_{\mathrm{II}}(w)=0 & \text { on } & \partial B_{r_{\varepsilon}} \\
w=0 & \text { on } & \partial \Omega_{\mathcal{A}} .
\end{array}\right.
}
In addition $\left\|\Gamma_{\varepsilon, \mathcal{A}}(f)\right\|_{2, \alpha, \nu} \leq c\|f\|_{0, \alpha, \nu-2}$, for some constant $c>0$ independent of $\kappa, \alpha, \varepsilon$ and $\mathcal{A}$ such that $\|\mathcal{A}\| \leq \kappa r_{\varepsilon}^{2}$.
\end{prop}
\begin{proof}
Define $\ti{\Lam}_{\veps,\calA}:\mathcal{C}_{\nu}^{2, \alpha}\left(\bar{\Omega} \backslash B_{r_{\varepsilon}}\right)\to \mathcal{C}_{\nu-2}^{0, \alpha}\left(\bar{\Omega} \backslash B_{r_{\varepsilon}}\right)$ by
\EQ{
\ti{\Lambda}_{\varepsilon, \mathcal{A}}\left(w \circ \Theta_{\mathcal{A}}\right) \equiv\left(\Lambda_{\varepsilon, \mathcal{A}} w\right) \circ \Theta_{\mathcal{A}},
}
where $\Theta_{\mathcal{A}}: \bar{\Omega} \rightarrow \overline{\Omega_{\mathcal{A}}}$ is a $\mathcal{C}^{2, \alpha}$ diffeomorphism such that $\Theta_{\mathcal{A}}(x)=x$ in $B_{r_{0} / 4}$ and $\left\|\Theta_{\mathcal{A}}-\mathrm{I}\right\|_{\mathcal{C}^{2, \alpha}} \leq c r_{\varepsilon}$ for some constant $c>0$ depending on $\kappa$ but independent of $\varepsilon$. Then using \eqref{eqn:7.1} we get
\EQ{
\|(\Lambda_{u}-\ti{\Lambda}_{\varepsilon, \mathcal{A}}) w\|_{0, \alpha, \nu-2} \leq c_{\kappa} r_{\varepsilon}^{2 / 3}\|w\|_{2, \alpha, \nu} .
}
It is now easy to see that, provided $\varepsilon$ is chosen small enough, and granted (A.5), the result follows from a simple perturbation argument.
\end{proof}
Given any $\kappa>0$, $\|\calA\|\leq \kappa r_{\veps}^2$ and $h_{\mathrm{II}}=\sum_{j\geq n+1} h_j e_j\in \pi_{\mathrm{II}}(C^{2,\alp}(\bbS^{n-1}))$ we define the function $w_0$ on the region $\Omega_{\calA}\setminus B_{r_\veps}$ as 
\EQ{
w_{0} \equiv \lambda\left(\frac{2 r_{0}-8 r}{r_{0}}\right) \sum_{j \geq n+1}\left(\frac{r}{r_{\varepsilon}}\right)^{\frac{2-n}{2}-\gamma_{j}} h_{j} e_{j}.
}
Note that $\Delta w_0 \equiv 0$ on $B_{r_0/2}\setminus B_{r_\veps}$. Then from Proposition \ref{prop:4.5} we see that 
\EQ{
\left\|w_{0}\right\|_{2, \alpha,-n} \leq c r_{\varepsilon}^{n}\|h\|_{2, \alpha},
}
for some constant $c>0$ which does not depend on $\varepsilon$. Then defining $w \equiv-\Gamma_{\varepsilon, \mathcal{A}}\left(\Lambda_{\varepsilon, \mathcal{A}} w_{0}\right)+w_{0}$ we see that $w$ solves,
\EQ{
\begin{cases}\Lambda_{\varepsilon, \mathcal{A}} w=0 & \text { in } \quad \Omega_{\mathcal{A}} \backslash B_{r_{\varepsilon}} \\ \pi_{\mathrm{II}}(w)=h_{\mathrm{II}}\left(\cdot / r_{\varepsilon}\right) & \text { on } \quad \partial B_{r_{\varepsilon}} \\ w=0 & \text { on } \quad \partial \Omega_{\mathcal{A}} .\end{cases}
}
This yields a mapping between $h_{\mathrm{II}}$ and $w$ defined as follows:
\EQ{\label{defn:pi_eps}
\Pi_{\varepsilon, \mathcal{A}}: h_{\mathrm{II}} \in \pi_{\mathrm{II}}\left(\mathcal{C}^{2, \alpha}\left(S^{n-1}\right)\right) \longrightarrow w \in \mathcal{C}_{\nu}^{2, \alpha}\left(\overline{\Omega_{\mathcal{A}}} \backslash B_{r_{\varepsilon}}\right),
}
with estimates  
\begin{align}\label{eqn:7.2}
    \|\Pi_{\varepsilon, \mathcal{A}}\left(h_{\mathrm{II}}\right)\left\|_{2, \alpha, \nu} \leq c_{\kappa} r_{\veps}^{-\nu}\right\| h_{\mathrm{II}} \|_{2, \alpha}
\end{align}
for any $\nu \in(-n, 1-n)$
We can now state the counterpart of Proposition \ref{prop:4.6}.
\begin{prop}\label{prop:7.2}
Assume that (A.1)--(A.5) hold. Fix $\nu \in(-n, 1-n)$ and $\alpha \in(0,1)$. Then, for all $\kappa>0$ there exist $c_{\kappa}>0$ and $\varepsilon_{0}>0$ such that, for all $\varepsilon \in\left(0, \varepsilon_{0}\right]$, we have
\EQ{
\left\|r_{\varepsilon} \partial_{r} \Pi_{\varepsilon, \mathcal{A}}\left(h_{\mathrm{II}}\right)\left(r_{\varepsilon} \theta\right)+\frac{n-2}{2} h_{\mathrm{II}}+D_{\theta} h_{\mathrm{II}}\right\|_{1, \alpha} \leq c_{\kappa}\left(r_{\varepsilon}^{n+\nu}+r_{\varepsilon}^{2 / 3}\right)\left\|h_{\mathrm{II}}\right\|_{2, \alpha} .
}
\end{prop}
\begin{proof}
See the argument in the proof of Proposition \ref{prop:4.6}. 
\end{proof}

\subsection{Existence of infinite dimensional minimal graphs over $\Sigma_{\veps,\calA}$}
The goal of this section is to prove the existence of an infinite-dimensional family of minimal hypersurfaces that are graphs over $\Sigma_{\veps,\calA}$ provided $\Sigma_0$ is itself minimal. This family of hypersurfaces will be parameterized by its Cauchy data. Throughout this section, we will assume that the assumptions (A.1)-(A.5) hold and that $\Sigma_0$ is a minimal hypersurface in $\R^{n+1}$.
\newline
We keep the notations of the last section and assume from now on that (A.1)-(A.5) hold. We will also assume that $\Sigma_{0}$ is a minimal hypersurface.
\subsubsection{Minimal Surfaces close to $\Sigma_{\veps,\calA}$}
The mean curvature of a surface close to $\Sigma_{\veps,\calA}$ can be expressed as a graph 
\EQ{
\overline{\Omega_{\mathcal{A}}} \backslash B_{r_{\varepsilon}} \ni x \longrightarrow\left(x, u(x)+w_{\varepsilon, \mathcal{A}}(x)+w(x)\right),
}
for some real-valued function $w$ with mean curvature 
\EQ{
H=H_{\varepsilon, \mathcal{A}}+\Lambda_{\varepsilon, \mathcal{A}} w-\operatorname{div} \mathcal{Q}_{\varepsilon, \mathcal{A}}(w),
}
where $H_{\varepsilon, \mathcal{A}}$ is the mean curvature of the hypersurface $\Sigma_{\varepsilon, \mathcal{A}}$ and where $\mathcal{Q}_{\varepsilon, \mathcal{A}}(w)$ represents all the nonlinear terms when one expands using \eqref{eqn:6.5} with $w$ replaced by $w_{\veps,\calA}+w$. Then, given boundary data $h_{\mathrm{II}} \in \pi_{\mathrm{II}}\left(\mathcal{C}^{2, \alpha}\left(\bbS^{n-1}\right)\right)$ our goal is to find a minimal graph $w$ close to $\Sigma_{\veps,\calA}$ or $H\equiv 0$. In other words, we want to solve the following problem
\EQ{\label{eqn:8.1}
\begin{cases}\Lambda_{\varepsilon, \mathcal{A}} w=-H_{\varepsilon, \mathcal{A}}+\operatorname{div} \mathcal{Q}_{\varepsilon, \mathcal{A}}(w) & \text { in } \quad \Omega_{\mathcal{A}} \backslash B_{r_{\varepsilon}} \\ \pi_{\mathrm{II}}\left(u+w_{\veps,\calA}+w\right)=h_{\mathrm{II}}\left(\cdot / r_{\varepsilon}\right) & \text { on } \quad \partial B_{r_{\varepsilon}} \\ w=0 & \text { on } \quad \partial \Omega_{\mathcal{A}} \end{cases}.
}
Given any $h_{\mathrm{II}} \in \pi_{\mathrm{II}}\left(\mathcal{C}^{2, \alpha}\left(\bbS^{n-1}\right)\right)$ with $\left\|h_{\mathrm{II}}\right\|_{2, \alpha} \leq \kappa r_{\varepsilon}^{2}$, we define $\tilde{w}$ on $\partial B_{r_{\veps}}$ such that
\EQ{
\tilde{w} &=\Pi_{\varepsilon, \mathcal{A}}\left(h_{\mathrm{II}}-\pi_{\mathrm{II}}(u+w_{\veps,\calA})(r_{\veps}\cdot) \right)-\Gamma_{\varepsilon, \mathcal{A}}\left(H_{\varepsilon, \mathcal{A}}\right)\\
&= \Pi_{\varepsilon, \mathcal{A}}\left(h_{\mathrm{II}}-\pi_{\mathrm{II}} \bar{w}_{\varepsilon, \mathcal{A}}\left(r_{\varepsilon} \cdot\right)\right)-\Gamma_{\varepsilon, \mathcal{A}}\left(H_{\varepsilon, \mathcal{A}}\right),
}
where $\Pi_{\veps,\calA}$ as in \eqref{defn:pi_eps} and $\bar{w}_{\veps,\calA}$ is defined in Proposition \ref{prop:6.3} and we used $\pi_{\mathrm{II}}\left(u+w_{\varepsilon, \mathcal{A}}\right)=\pi_{\mathrm{II}} \bar{w}_{\varepsilon, \mathcal{A}}$ on $\partial B_{r_{\varepsilon}}$ since the rigid motions generate Jacobi fields giving rise to the lower eigenmodes that are killed by the map $\pi_{\mathrm{II}}.$
\newline
Thus if we set $w=\tilde{w}+v$, then the task of solving \eqref{eqn:8.1} reduces to find $v \in \mathcal{C}_{\nu}^{2, \alpha}\left(\Sigma_{\varepsilon, \mathcal{A}}\right)$ such that
\EQ{
\begin{cases}\Lambda_{\varepsilon, \mathcal{A}} v=\operatorname{div} \mathcal{Q}_{\varepsilon, \mathcal{A}}(\tilde{w}+v) & \text { in } \quad \Omega_{\mathcal{A}} \backslash B_{r_{\varepsilon}} \\ \pi_{\mathrm{II}}(v)=0 & \text { on } \quad \partial B_{r_{\varepsilon}} \\ v=0 & \text { on } \quad \partial \Omega_{\mathcal{A}} \end{cases},
}
which can be done by applying a fixed-point mapping argument to the map 
\EQ{
\mathcal{M}_{\varepsilon, \mathcal{A}}(v)=\Gamma_{\varepsilon, \mathcal{A}}\left(\mathcal{Q}_{\varepsilon, \mathcal{A}}(\tilde{w}+v)\right).
}
The following proposition shows that this can be done for a careful choice of parameters. 
\begin{prop}
Assume that $\nu \in(-n, 1-n)$ (or $\nu \in (-8 / 3,-2)$  when $n=3$) and that $\alpha \in (0,1)$ are fixed. For all $\kappa>0$, there exist $c_{\kappa}>0$ and $\varepsilon_{0}>0$ such that, for all $\varepsilon \in\left(0, \varepsilon_{0}\right]$, if $h_{\mathrm{II}} \in \pi_{\mathrm{II}}\left(\mathcal{C}^{2, \alpha}\left(\bbS^{n-1}\right)\right)$ is fixed with
\EQ{
\left\|h_{\mathrm{II}}\right\|_{2, \alpha} \leq \kappa r_{\varepsilon}^{2}
}
then $\mathcal{M}_{\varepsilon, \mathcal{A}}$ is a contraction mapping on the ball
\EQ{
\calB \equiv\left\{v:\|v\|_{2, \alpha, \nu} \leq c_{\kappa} r_{\varepsilon}^{10 / 3-\nu}\right\}
}
and thus has a unique fixed point in this ball. 
\end{prop}
\begin{proof}
We first prove that
\EQ{
\left\|\mathcal{M}_{\varepsilon, \mathcal{A}}(0)\right\|_{2, \alpha, \nu} \leq \frac{c_{\kappa}}{2} r_{\varepsilon}^{10 / 3-\nu}
}
for some constant $c_{\kappa}>0$. By definition of $\calM_{\veps,\calA}$ it suffices to estimate $\|\Gamma_{\veps,\calA}(\calQ_{\veps,\calA}(\ti w))\|_{2,\alp,\nu}$ and by Proposition \ref{prop:7.1}, this reduces to estimating
\EQ{
\|\Gamma_{\veps,\calA}(\calQ_{\veps,\calA}(\ti w))\|_{2,\alp,\nu} \leq c \|\calQ_{\veps,\calA}(\ti w)\|_{0,\alp,\nu-2}
}
for some constant $c>0.$ Now using \eqref{eqn:6.5} with $w$ replaced by $w_{\varepsilon, \mathcal{A}}+w$ we can write
\EQ{
\mathcal{Q}_{\varepsilon, \mathcal{A}}(w) \equiv\left(r_{\veps}+\varepsilon r_{\veps} r^{-n}\right) Q_{\varepsilon, \mathcal{A}}^{\prime}(\nabla w)+Q_{\varepsilon, \mathcal{A}}^{\prime \prime}(\nabla w),
}
where $q \rightarrow Q_{\varepsilon, \mathcal{A}}^{\prime}(q)$ is homogeneous of degree 2 and $q \rightarrow Q_{\varepsilon, \mathcal{A}}^{\prime \prime}(q)$ satisfies
\EQ{
Q_{\varepsilon, \mathcal{A}}^{\prime \prime}(0)=0 \quad \nabla_{q} Q_{\varepsilon, \mathcal{A}}^{\prime \prime}(0)=0 \quad \text { and } \quad \nabla_{q q}^{2} Q_{\varepsilon, \mathcal{A}}^{\prime \prime}(0)=0
}
and the power of $r^{-n}$ comes from Proposition \ref{prop:6.3}. Regarding $\ti w$, we can use \eqref{eqn:7.2} to get
\EQ{\label{eqn:8.2}
\left\|\Pi_{\varepsilon, \mathcal{A}}\left(h_{\mathrm{II}}-\pi_{\mathrm{II}}\left(\bar{w}_{\varepsilon, \mathcal{A}}\right)\right)\right\|_{2, \alpha, \nu} \leq c r_{\varepsilon}^{-\nu}\left\|h_{\mathrm{II}}-\pi_{\mathrm{II}}\left(\bar{w}_{\varepsilon, \mathcal{A}}\right)\right\|_{2, \alpha},
}
and
Proposition \ref{prop:6.3} to get
\EQ{\label{eqn:8.3}
\left\|\pi_{\mathrm{II}}\left(\bar{w}_{\varepsilon, \mathcal{A}}\right)\right\|_{2, \alpha} &\leq c r_{\varepsilon}^{2}
}
and Propositions \ref{prop:6.1} and \ref{prop:7.1} to get
\EQ{
\left\|\Gamma_{\varepsilon, \mathcal{A}}\left(H_{\varepsilon, \mathcal{A}}\right)\right\|_{2, \alpha, \nu} \leq c r_{\varepsilon}^{2-\nu}
}
for some constant $c>0$ which does not depend on $\kappa$, nor on $\mathcal{A}$, provided $\varepsilon$ is taken small enough. Note that the restriction on the parameter $\nu$ when $n=3$ comes from the above estimate since the upper bound of $r_\veps^{2-\nu}$ only holds when $\nu\in (-8/3,-2)$ when $n=3$. Thus, combining the previous displays, we get that
\EQ{
\|\ti w\|_{2,\alp,\nu}\leq c_{\kappa} r_\veps^{2-\nu},
}
for a constant $c_{\kappa}>0$ that depends on $\kappa>0.$ Therefore, using that $\veps\lesssim  r_{\veps}^{(3n-2)/3}$ we get
\EQ{
\left\|\operatorname{div}\left(\left(r_{\veps}+\varepsilon r_{\veps} r^{-n}\right) Q_{\varepsilon, \mathcal{A}}^{\prime}(\tilde{w})\right)\right\|_{0, \alpha, \nu-2} \leq c'_{\kappa} \veps r_\veps r^{-n-1} r_{\veps}^{4-\nu} \leq \tilde{c}_{\kappa} r_{\veps}^{(3n-2)/3+1-n-1+4-\nu} \leq \tilde{c}_{\kappa} r_{\veps}^{10/3-\nu}
}
while
\EQ{
\left\|\operatorname{div}\left(Q_{\varepsilon, \mathcal{A}}^{\prime \prime}(\tilde{w})\right)\right\|_{0, \alpha, \nu-2} \leq \tilde{c}_{\kappa} r_{\varepsilon}^{4-\nu},
}
provided $\varepsilon$ is chosen small enough, say $\varepsilon \in\left(0, \varepsilon_{0}\right]$, where the constant $\tilde{c}_{\kappa}>0$ depends on $\kappa$. The existence of $c_{\kappa}$ follows from Proposition \ref{prop:7.1}. for some constant $c$ which is independent of $\kappa$, provided $\varepsilon$ is chosen small enough. This shows that $0\in \calB.$ To finish the argument, we also need to show that if $v_1,v_2\in \calB$ then
\EQ{
\left\|\mathcal{M}_{\varepsilon, \mathcal{A}}\left(v_{2}\right)-\mathcal{M}_{\varepsilon, \mathcal{A}}\left(v_{1}\right)\right\|_{2, \alpha, \nu} \leq \frac{1}{2}\left\|v_{2}-v_{1}\right\|_{2, \alpha, \nu},
}
provided $v_{1}$ and $v_{2}$ belong to $B$. The argument for this follows a similar strategy, and the constant $1/2$ is obtained by choosing $\veps_0>0$ small enough.
\end{proof}
Thus, we have obtained a minimal hypersurface close to $\Sigma_{\veps,\calA}$, with two boundaries, one of which (up to rigid motions) is the boundary of $\Sigma_0$ (called the \textit{outer} boundary), while the other one will be called as the \textit{inner} boundary. 
\newline
Translating this minimal hypersurface along the $x_{n+1}$ axis by an amount $-\varepsilon r_{\varepsilon}^{2-n} /(n-2)$ gives us a new hypersurface denoted by $\Sigma_{\varepsilon}\left(\mathcal{A}, h_{\mathrm{II}}\right)$. We can parametrize this surface as 
\EQ{
\overline{\Omega_{\mathcal{A}}} \backslash B_{r_{\varepsilon}} \ni x \longrightarrow\left(x, V_{\varepsilon, \mathcal{A}, h_{\mathrm{II}}}(x)\right) \in \Sigma_{\varepsilon}\left(\mathcal{A}, h_{\mathrm{II}}\right) .
}
\begin{defn}
We define the second Cauchy data map as
\EQ{
\mathcal{T}_{\varepsilon}\left(\mathcal{A}, h_{\mathrm{II}}\right)(\theta) \equiv\left(V_{\varepsilon, \mathcal{A}, h_{\mathrm{II}}}\left(r_{\varepsilon} \theta\right), r_{\varepsilon} \partial_{r} V_{\varepsilon, \mathcal{A}, h_{\mathrm{II}}}\left(r_{\varepsilon} \theta\right)\right).
}  
To define the domain of the map $\calT_{\veps}$, we first denote $\calF$ as the collection of tuples $(\calA,w)$ where
\EQ{
\calF \ni (\calA,w) \in \mathbb{R}^{n} \times \mathbb{R}^{n} \times \mathbb{R} \times \mathbb{R} \times \pi_{\mathrm{II}}\left(\mathcal{C}^{2, \alpha}\left(\bbS^{n-1}\right)\right)
}
with the natural norm
\EQ{
\|(\mathcal{A}, w)\|_{\mathcal{F}} \equiv\|\mathcal{A}\|+\|w\|_{2, \alpha} .
}
The domain of $\mathcal{T}_{\varepsilon}$ is then simply a subset of $\mathcal{F}$.
\end{defn}
Using Proposition \ref{prop:6.3} and the definition of $w$ as a solution of \eqref{eqn:8.1} we can explicitly write $\calT_{\veps}(\calA,h_{\mathrm{II}})$ as 
\EQ{
\mathcal{T}_{\varepsilon}\left(\mathcal{A}, h_{\mathrm{II}}\right)= & \left(\left(w_{\varepsilon, \mathcal{A}}^{0}+\bar{w}_{\varepsilon, \mathcal{A}}+w\right)\left(r_{\varepsilon} \cdot\right),-\varepsilon r_{\varepsilon}^{2-n}\right. \\
& \left.+r_{\varepsilon} \partial_{r}\left(w_{\varepsilon, \mathcal{A}}^{0}+\bar{w}_{\varepsilon, \mathcal{A}}+w\right)\left(r_{\varepsilon} \cdot\right)\right), 
}
where
\EQ{
w_{\varepsilon, \mathcal{A}}^{0}(x) \equiv \frac{e}{n-2} r^{2-n}+d+R \cdot x+\varepsilon r^{-n} T \cdot x.
}
Then we define 
\EQ{
& \mathcal{T}_{0}\left(\mathcal{A}, h_{\mathrm{II}}\right):=\left(w_{\varepsilon, \mathcal{A}}^{0}\left(r_{\varepsilon} \cdot\right)+h_{\mathrm{II}},-\varepsilon r_{\varepsilon}^{2-n}+r_{\varepsilon} \partial_{r} w_{\varepsilon, \mathcal{A}}^{0}\left(r_{\varepsilon} \cdot\right)-\frac{n-2}{2} h_{\mathrm{II}}-D_{\theta} h_{\mathrm{II}}\right).
}
\begin{prop}\label{prop:8.2}
The mappings $\mathcal{T}_{\varepsilon}$ and $\mathcal{T}_{0}$ are continuous. Furthermore, there exists $c>0$ and, for all $\kappa>0$, there exists $\varepsilon_{0}>0$ such that, for all $\varepsilon \in\left(0, \varepsilon_{0}\right]$, we have the estimate
\EQ{
\left\|\left(\mathcal{T}_{\varepsilon}-\mathcal{T}_{0}\right)\left(\mathcal{A}, h_{\mathrm{II}}\right)\right\|_{\mathcal{C}^{2, \alpha} \times \mathcal{C}^{1, \alpha}} \leq c r_{\varepsilon}^{2}
}
Again, it is important in the last Proposition that the constant $c>0$ does not depend on $\kappa$.
\end{prop}
\begin{proof}
The proof of this Proposition is the same as the proof of Proposition \ref{prop:5.1}.
\end{proof}

\subsection{Gluing construction}
In this subsection, we use the results established earlier to complete the gluing procedure. The construction of our minimal hypersurface will proceed inductively. Thus, suppose that we start with an orientable, embedded, minimal hypersurface $\Sigma^n \subset \R^{n+1}$ with $k$ planar ends denoted by $E_1,\ldots, E_k$ that are parallel to the $\{x_{n+1}=0\}$ plane. Each end $E_i$ can be written as a normal graph over an appropriately scaled half-catenoid as follows:
\EQ{\label{eqn:sigma-graph}
\left[S_{i},+\infty\right) \times \bbS^{n-1} \ni(s, \theta) \longrightarrow a_{i} X_{0}(s, \theta)+w_{i}(s, \theta) \phi^{\frac{2-n}{2}}(s) N_{0}(s, \theta) \in E_{i},
}
where $a_{i} \in(0,+\infty)$ and where $w_{i} \in \mathcal{C}_{\delta}^{2, \alpha}\left(\left[S_{j},+\infty\right) \times S^{n-1}\right)$ for any $\delta \in \left(-\frac{2+n}{2},-\frac{n}{2}\right)$. Similar to Section \ref{sec:4}, we denote the $2(n+1)$ linearly independent Jacobi fields by
\EQ{
\Psi_{i}^{j, \pm} \quad \text { for } \quad j=0, \ldots, n \quad \text { and } \quad i=1, \ldots, k .
}
We can split $\Sigma$ into a compact piece $\Sigma^c$ and $\Sigma\setminus \Sigma^c= \cup_{i=1}^k E_i$, where for each end $\Sigma^c \cap E_i$ is diffeomorphic to $[0,1] \times S^{n-1}$. Given a point $p\in \R^{n+1}$, let $r_0>$ be small enough such that $\Sigma$ can be written as a graph over a ball $B_{4r_0}(p)$ with a heigh function $u_0.$ Define $\Sigma_{r_0}=\Sigma \setminus B_{r_0}(p)$ and let $\Sigma^c_r$ denote the graph of $u_0$ over $B_r$ for all $r\le 4r_0.$ \\
The idea will be construct a family of nearby minimal hypersurfaces $\Sigma_{r_0}(h)$ and $\Sigma^c_{r_0}(h)$ that are parameterized by the boundary value $h.$  Using the results in the previous subsections applied to $\Sigma^c_{r_0}(h)$, we can produce a family of minimal hypersurfaces $\Sigma^c_{r_0}(h,\calA, h_{\mathrm{II}})$. Then the goal will be to find boundary data $h,h_{\mathrm{II}}$ and parameters $\calA$ such that $\Sigma_{r_0}(h)$, $\Sigma^c_{r_0}(h,\calA, h_{\mathrm{II}})$ and the perturbed catenoid $C_{\veps}(h_{\mathrm{II}})$ agree at the boundaries.
\newline
Note that the resulting hypersurface will have $k+1$ planar ends, where the new end comes from gluing the catenoidal piece $C_{\veps}(h_{\mathrm{II}}).$ To guarantee that this end is parallel to the $\{x_{n+1}=0\}$ plane, we will chose the point $p\in \R^{n+1}$ far enough such that assumption (A.2) holds, i.e. upon translating \(p\) to the origin, that $|\nabla u(0)| < r_\veps$.

\subsubsection{Construction of $\Sigma_{r_0}(h)$ and $\Sigma^c_{r_0}(h)$}
In order to construct a family of nearby minimal hypersurfaces $\Sigma_{r_0}(h)$ and $\Sigma^c_{r_0}(h)$ given some boundary data $h$, we will apply the inverse function theorem to suitable operators on appropriate function spaces. While on compact domains this can be done as in the previous sections, on non-compact domains one needs a suitable modification. To this end, we define the following function space:
\begin{defn}\label{defn:9.1}
The function space $\mathcal{E}_{\mu}^{k, \alpha}(\Sigma)$ is defined to be the space of all functions $w \in \mathcal{C}^{k, \alpha}(\Sigma)$ for which the following norm is finite
\EQ{
|w|_{k, \alpha, \delta} \equiv \sum_{i=1}^{k}\left\|w_{|E_i}\right\|_{k, \alpha, \delta}+\left\|w_{|\Sigma^c}\right\|_{k, \alpha, \Sigma^c} .
}
where $\|\cdot \|_{k, \alpha, \delta}$ is the norm defined in Definition \ref{defn:4.1}. Notice that we have identified the function $w$ on $E_{i}$ with a function on $\left[S_{i},+\infty\right) \times S^{n-1}$ via the graph representation of $E_{i}$.    
\end{defn}
Let $\mathcal{L}_\Sigma=\phi^{\frac{2-n}{2}}\mathcal{L}_{\Sigma,0} \phi^{\frac{2-n}{2}}$ denote the conjugate of the linearized mean curvature operator with $\phi$ as in \eqref{eqn:sigma-graph}. Note that while $\phi$ is defined on each end $E_i$, it can be extended into a global smooth function $\phi>0$ on $\Sigma.$ At the core of this section, we will be interested in understanding the invertibility of the operator $\calL_{\Sig}.$ 
For $r_0>0$ small enough, define a modified transverse normal vector field on $\Sigma$ as 
\EQ{
N'(x) := \left\{\begin{array}{cc}
     (0,\ldots, 1) &  \text{in }\Sigma^c_{2r_0}\\
     N_0(x)& \text{ on } \Sigma_{4r_0} 
\end{array}\right.
}
where $N_0$ is the normal vector field of $\Sigma$ and denote $\mathcal{L}'_{\Sigma}$ as the linearized mean curvature operator obtained by perturbing along $N'.$ Then observe that
\EQ{
\calL_{\Sigma}':= \left\{\begin{array}{cc}
    \Lam_{u_0} & \text{ in } \Sigma_{2r_0}^c\\ 
    \calL_{\Sig} & \text{ in }\Sigma_{4r_0} 
\end{array} \right.
}
The following key Lemma will allow us to find $\Sigma_{r_0}(h)$ and $\Sigma^c_{r_0}(h)$.
\begin{lem}\label{lem:invertability-of-L_M}
Assume that $\Sigma$ is non-degenerate, in the sense that, 
\EQ{
\mathcal{L}_{\Sigma}: \mathcal{E}_{\delta}^{2, \alpha}(\Sigma) \longrightarrow \mathcal{E}_{\delta}^{0, \alpha}(\Sigma),
}
is injective for all $\delta\in (-\infty,-\frac{n}{2}).$ Define the deficiency space as 
\EQ{
    \mathcal{K} := \oplus_{i=1, \ldots, k} \operatorname{Span}\left\{\lambda\left(\cdot-S_{i}\right) \Psi_{i}^{j, \pm}: j=0, \ldots, n\right\}.
}
There exists $r_0 \in (0,1)$ small enough such that the following holds. 
\begin{enumerate}
    \item For any fixed $\delta \in\left(-\frac{2+n}{2},-\frac{n}{2}\right)$, the operator
    \EQ{
    \mathcal{L}_{\Sigma}^{'}: \mathcal{E}_{\delta}^{2, \alpha}(\Sigma) \oplus \mathcal{K}_{1}\to \mathcal{E}_{\delta}^{0, \alpha}(\Sigma)
    }
    is an isomorphism for some $k(n+1)$ dimensional subspace $\mathcal{K}_1\subset \mathcal{K}$ such that $\calK=\calK_0\oplus \calK_1$, and $\calK_0$ is a $k(n+1)$ dimensional subspace such that $\operatorname{ker}({\mathcal{L}_{\Sigma}}_{|\calE^{2,\alp}_{-\delta}(\Sigma)})= \calK_0 \oplus \calE^{2,\alp}_{\delta}(\Sigma)$.
    \item For any fixed $\delta \in (-\frac{2+n}{2},-\frac{n}{2})$, the operator 
    \EQ{
    \mathcal{L}_{\Sigma}^{'}:[\mathcal{E}_{\delta}^{2, \alpha}(\overline{\Sigma_{r_{0}}}) \oplus \mathcal{K}_{1}]^{'}\to \mathcal{E}_{\delta}^{0, \alpha}
    }
    is an isomorphism where the domain is defined as 
    \EQ{
    [\mathcal{E}_{\delta}^{2, \alpha}\left(\overline{\Sigma_{r_{0}}}\right) \oplus \mathcal{K}_{1}]^{'}:= \{w \in \mathcal{E}_{\delta}^{2, \alpha}(\overline{\Sigma_{r_{0}}}) \oplus \mathcal{K}_{1}: w=0 \text { on }  \partial \Sigma_{r_{0}}\} .
    }
    \item The operator
    \EQ{
    \mathcal{L}_{\Sigma}^{'}=\Lambda_{u_{0}}:[\mathcal{C}^{2, \alpha}(\overline{B_{r_{0}}})]'\to \mathcal{C}^{0, \alpha}(\overline{B_{r_{0}}})
    } is an isomorphism.
    \item For some fixed $\nu \in(-n, 1-n)$ and any $0<r<r_{0} / 2$, there exists a mapping 
    \EQ{
    \Gamma_{u_{0}, r}:\mathcal{C}_{\nu-2}^{0, \alpha}(\overline{B_{r_{0}}} \backslash B_{r})\to [\mathcal{C}_{\nu}^{2, \alpha}(\overline{B_{r_{0}}} \backslash B_{r})]^{'}
    }
    such that $\Lambda_{u_{0}} \circ \Gamma_{u_{0}, r}=\operatorname{Id}$, with operator norm of $\Gamma_{u_{0}, r}$ bounded independently of $r$ for $r<r_0/2$ small enough.
\end{enumerate}
\end{lem}

\begin{proof}
Since $\calL_{\Sigma}$ is injective, by duality we deduce that $\calL_{\Sigma}:\calE^{2,\alp}_{-\delta}(\Sigma)\to \calE^{0,\alp}_{-\delta}(\Sigma)$ is surjective. Next, by arguing in the same way as in the proof of of the Linear Decomposition Lemma in \cite{yamabe-metric,cmc-surface} we can first show that given any $f\in \calE^{0,\alp}_{\delta}$, there exists a solution $w\in \calE_{-\delta}^{2,\alp}$ such that $\calL_{\Sigma}w=f$ and more importantly, decomposes as $w=v+\phi$, where $v\in \calE^{2,\alp}_{\delta}$ and $\phi \in \calK$, where $\calK$ is defined in the statement of the above Lemma. Thus, $w$ can be decomposed into a piece that has the same decay as $f$. This implies that,
\EQ{\label{eqn:mapping}
\calL_\Sigma:\calE^{2,\alp}_{\delta}(\Sigma) \oplus \calK \to \calE^{0,\alp}_{\delta}(\Sigma)
}
is surjective with a bounded kernel 
\EQ{
\calB:=\{u\in \calE^{2,\alp}_{\delta}(\Sigma)\oplus \calK:\calL_{\Sigma} u =0\}.
}
To make this into an isomorphism, we remove the kernel in the above mapping \eqref{eqn:mapping}. Now consider the projection map $\Pi:\calB\to \calK$. Note that since $\calL_\Sigma$ is injective on $\calE_{\delta}^{2,\alp}(\Sigma)$, this projection map is also injective since if $u,v\in \calB$ and $\Pi(u)=\Pi(v)$ then $\calL_{\Sigma}(u-v)=0$ and $(u-v)\in \calE^{2,\alp}_{\delta}(\Sigma)$. Since $\calL_\Sigma$ is injective on $\calE^{2,\alp}_\delta(\Sigma)$ we deduce that $u=v.$ Therefore the map $\Pi$ is also injective and consequently we will identify $u\in \calB$ with its projection onto $\calK.$ If we denote the image $\calK_0=\Pi(\calB)$ and the orthogonal space $\calK_1$ such that $\calK=\calK_0\oplus \calK_1$, then the map
\EQ{\
\calL_\Sigma:\calE^{2,\alp}_{\delta}(\Sigma) \oplus \calK_1 \to \calE^{0,\alp}_{\delta}(\Sigma)
}
is an isomorphism since,
\begin{itemize}
    \item if $\calL_\Sigma u= 0$ for $u=v+\phi_1$ where $v\in \calE^{2,\alp}_\delta$ and $\phi_1\in \calK_1$, then $v+\phi_1\in \calB$. However, $\Pi(\phi_1)=0$ and therefore we deduce that $\phi_1\equiv 0$. Then $u=0$ by injectivity of $\calL_\Sigma$ on $\calE^{2,\alp}_\delta(\Sigma).$ This proves that the map is injective.
    \item On the other hand, if $f\neq 0$, $f\in \calE^{0,\alp}_\delta(\Sigma)$, then there exists $w=v+\phi_0+\phi_1\in \calE^{2,\alp}_\delta\oplus \calK_0\oplus \calK_1$ such that $\calL_\Sigma w = f.$ Now, writing $\phi_0 = \Pi(u_0)$ for some $u_0\in \calB$, with $u_0=\phi_0 + \tilde{u}_0$, where $\phi_0\in \calK_0$ and $\tilde{u}_0\in \calE^{2,\alp}_\delta$ we deduce that $\tilde{w}=w-u_0$ satisfies 
    \EQ{
    \calL_\Sigma (\tilde{w}) = f-\calL_\Sigma u_0 = f.
    }
    Thus we have found a solution $\tilde{w}$ given $f$ such that
    \EQ{
    \tilde{w} = w-u_0  = v+\phi_0+\phi_1 - \phi_0 - \tilde{u}_0 = (v-\tilde{u}_0)+ \phi_1 \in \calE^{2,\alp}_{\delta} \oplus \calK_1,
    }
    which shows that the map is surjective.
\end{itemize}
Now the first three items follow from a standard perturbation argument for $r_0$ small enough. On the other hand, the proof of the fourth item follows from an application of Lemma \ref{lem:6.1}
\EQ{
\left\|\left(\Lambda_{u}-\Delta\right) w\right\|_{0, \alpha, \nu-2} \leq c r_{0}^{2}\|w\|_{2, \alpha, \nu}
}
and from the fact that given any $f \in \mathcal{C}_{\nu-2}^{0, \alpha}((\overline{B_{r_{0}}} \backslash B_{r}))$ the solution to the problem 
\EQ{
\left\{\begin{array}{lll}
\Delta w=f & \text { in } & B_{r_{0}} \backslash B_{r} \\
\pi_{\mathrm{II}} w=0 & \text { on } & \partial B_{r} \\
w=0 & \text { on } & \partial B_{r_{0}} .
\end{array}\right.
}
satisfies the estimate 
\EQ{
\left\|w\right\|_{2, \alpha, \nu} \leq c\|f\|_{0, \alpha, \nu-2}
}
for some constant $c>0$ independent of $r$. The proof of the above display is similar to the proof of Proposition \ref{prop:4.3}.
\end{proof}
\subsection{Matching the Cauchy data}
\subsubsection{Defining the Cauchy data for the outer boundary}
Recall Propositions \ref{prop:5.2} and \ref{prop:8.2}, where we defined the Cauchy data and simple Cauchy data mappings \(\calS_\veps, \calT_\veps, \calS_0\), and \(\calT_0\). These maps represented the Cauchy data for the boundary value problems at \(C_\veps(h_\mathrm{II})\) and the inner boundary of \(\Sigma^c_{r_0, \veps}(h_\mathrm{I}, \calA, h_\mathrm{II})\).

In order to glue \(\Sigma_{r_0}(h_\mathrm{I})\), \(\Sigma^c_{r_0, \veps}(h_\mathrm{I}, \calA, h_\mathrm{II})\), and \(C_\veps(h_\mathrm{II})\), we must also match the Cauchy data of \(\Sigma_{r_0}(h_\mathrm{I})\) with the outer boundary of \(\Sigma^c_{r_0, \veps}(h_\mathrm{I}, \calA, h_\mathrm{II})\). By construction, the surfaces \(\Sigma_{r_{0}}(h_\mathrm{I})\) and \(\Sigma^c_{r_{0}, \varepsilon}\left(h_\mathrm{I}, \mathcal{A}, h_{\mathrm{II}}\right)\) are graphs over the $x_{n+1}=0$ hyperplane near their boundary given in terms of the perturbation functions \(w\) we described earlier. Concretely we can say \(\Sigma_{r_0}(h_\mathrm{I})\) is the graph of
\begin{equation}
    x \in B_{2 r_{0}} \backslash B_{r_{0}} \longrightarrow\left(x, u_{0}(x)+w_{h_\mathrm{I}}(x)\right)
\end{equation}
and \(\Sigma^c_{r_{0}, \varepsilon}\left(h_\mathrm{I}, \mathcal{A}, h_{\mathrm{II}}\right)\) is the graph
\begin{equation}
    x \in B_{r_{0}} \backslash B_{r_{0} / 2} \longrightarrow\left(x, u_{0}(x)+\tilde{w}_{h_\mathrm{I}, \mathcal{A}, h_{\mathrm{II}}, \varepsilon}(x)\right).
\end{equation}
By construction, the 0th order boundary data are \(B_{r_0}\) matches for both surfaces (i.e., both are the graph of \(u_0 + h_\mathrm{I}\) over \(B_{r_0}\)), so the Cauchy data map will contain only first order data:
\begin{defn}
The Cauchy data for the outer boundary problem are
    \begin{equation}
        \mathcal{U}_{\varepsilon}\left(h_\mathrm{I}, \mathcal{A}, h_{\mathrm{II}}\right) = r_{0} \partial_{r}\left(w_{h_\mathrm{I}}\left(r_{0} \theta\right)-\tilde{w}_{h_\mathrm{I}, \mathcal{A}, h_{\mathrm{II}}, \varepsilon}\left(r_{0} \theta\right)\right) \in \mathcal{C}^{1, \alpha}\left(\partial B_{r_{0}}\right).
    \end{equation}
\end{defn}

Now, as we did before, we must produce a simplified Cauchy data map. We let \(\calL^*_\Sigma\) be the linearized mean curvature operator of \(\Sigma\) with respect to a locally vertical vector field, and we let \(\Lambda_{u_0}\) be the linearized mean curvature operator of \(N\) with respect to the vertical vector field.
\begin{defn}
    For \(h_\mathrm{I} \in C^{2, \alpha}(\partial B_{r_0})\), the simple Cauchy data map for the outer boundary problem is
    \begin{equation}
        \calU_0(h_\mathrm{I}) = r_0\partial_r\left(w_{h_\mathrm{I}}^0(r_0\theta) - \tilde{w}_{h_\mathrm{I}}^0(r_0\theta)\right) \in C^{1, \alpha}(\partial B_{r_0}),
    \end{equation}
    where
    \begin{equation}
        \begin{cases}\calL^*_\Sigma w_{h_\mathrm{I}}^0 = 0 & \text{in \(\Sigma_{r_0}\)},\\ w_{h_\mathrm{I}}^0 = h_\mathrm{I} & \text{on \(\partial B_{r_0}\),}\\ w_{h_\mathrm{I}}^0 \in \calE_\delta^{2,\alpha}(\Sigma) \oplus \calK_1\end{cases} \quad \text{and} \quad  \begin{cases}\Lambda_{u_0}\tilde{w}_{h_\mathrm{I}}^0 = 0 & \text{in \(B_{r_0}\)},\\ \tilde{w}_{h_\mathrm{I}}^0 = h_\mathrm{I} & \text{on \(\partial B_{r_0}\),}\\\tilde{w}_{h_\mathrm{I}}^0 \in C^{2, \alpha}(\overline{B_{r_0}})\end{cases}
    \end{equation}
\end{defn}
We now prove a proposition about the difference of \(\calU_\veps\) and \(\calU_0\) in the fashion of the previous Cauchy data mappings.
\begin{prop}\label{prop:bounding Uveps - U0}
    For \(h_\mathrm{I} \in C^{2, \alpha}(\partial B_{r_0})\), we have
    \begin{equation}
        \left\|\calU_\veps\left(h_\mathrm{I}, \mathcal{A}, h_{\mathrm{II}}\right)-\calU_0(h_\mathrm{I})\right\|_{\mathcal{C}^{1, \alpha}} \leq c\left(\|h_\mathrm{I}\|_{\mathcal{C}^{2, \alpha}}^{2}+r_{\varepsilon}^{n-2 / 3}\right).
    \end{equation}
    Where \(c\) does not depend on \(\veps\). Furthermore, \(\calU_0\) is an isomorphism from \(C^{2,\alpha}(\partial B_{r_0}) \mapsto C^{1, \alpha}(\partial B_{r_0})\).
\end{prop}
\begin{proof}
    The fact that \(\calU_0\) is an isomorphism follows from claim (1) in Lemma \ref{lem:invertability-of-L_M}. To prove the bound on the norm, first, we consider
    \begin{equation}
        \|r_0\partial_r(w_{h_\mathrm{I}} - w_{h_\mathrm{I}}^0)\|_{C^{1, \alpha}(\partial B_{r_0})}
    \end{equation}
    This Cauchy data comes from solving perturbation problems on \(\Sigma\). If \(H(w)\) is the mean curvature operator of \(\Sigma\), then
    \begin{equation}
        0 = H(w_{h_\mathrm{I}}) = \calL^*_\Sigma w_{h_\mathrm{I}} + Q(w_{h_\mathrm{I}})
    \end{equation}
    where \(Q\) is a quadratic term. It follows that
    \begin{equation}
        \calL_\Sigma^*(w_{h_\mathrm{I}} - w_{h_\mathrm{I}}^0) = -Q(w_{h_\mathrm{I}}),
    \end{equation}
    and so, from elliptic estimates, we can bound 
    \begin{equation}
        \|r_0\partial_r(w_{h_\mathrm{I}} - w_{h_\mathrm{I}}^0)\|_{C^{1, \alpha}(\partial B_{r_0})} \le c(\calL_\Sigma^*)\|h_\mathrm{I}\|^2_{C^{2, \alpha}(\partial B_{r_0})}.
    \end{equation}
    The constant \(c(\calL_\Sigma^*)\) will depend on \(\veps\), as the locally vertical vector field depends on \(\veps\), but we can still make the constant uniform for all \(\veps \in (0, \veps_0]\) for a sufficiently small \(\veps_0\).

    Now we consider the second half of the Cauchy data given by
    \begin{equation}
        \|r_0\partial_r(\tilde{w}_{h_\mathrm{I}, \mathcal{A}, h_{\mathrm{II}}, \varepsilon} - \tilde{w}_{h_\mathrm{I}}^0)\|_{C^{1, \alpha}(\partial B_{r_0})}.
    \end{equation}
    This data corresponds to the boundary value problem in \(N\). For this, we track how the boundary data is influenced by the inner boundary problem. First, \(\Sigma^c_{r_0}\) gets modified to \(\Sigma^c_{r_0, \veps}\) by adding \(c\veps\gamma_0\) for a constant \(c\). Rigid motions are applied, and \(\Sigma^c_{r_0, \veps}\) is perturbed. The perturbation \(w\) ultimately satisfies a bound \(\|w\|_{2, \alpha, \nu} \le cr_\veps^{2 - \nu}\) for \(\nu \in (-n, 1 - n)\), in particular showing that
    \begin{equation}
        \|w\|_{1, \alpha} \le cr_\veps^{2 - \nu},
    \end{equation}
    so that the Cauchy data of \(\Sigma^c_{r_0, \veps}\) gets perturbed by at most \(cr_\veps^{2 - \nu}\). So the Cauchy data is modified by \(c\veps \le c r_\veps^{n - 2/3}\) and \(cr_\veps^{2 - \nu} \le c r_\veps^{n - 2/3}\), and we end up with 
    \begin{equation}
        \|r_0\partial_r(\tilde{w}_{h_\mathrm{I}, \mathcal{A}, h_{\mathrm{II}}, \varepsilon} - \tilde{w}_{h_\mathrm{I}}^0)\|_{C^{1, \alpha}(\partial B_{r_0})} \le cr_\veps^{n - 2/3}.
    \end{equation}
\end{proof}

\subsubsection{The conglomerate Cauchy data mapping}
Now we put together the various Cauchy data maps.
\begin{defn}
    For \((h_\mathrm{I}, \calA, h_\mathrm{II}) \in C^{2, \alpha}(\partial B_{r_0}) \times \mathbb{R}^{2n + 2} \times \pi_\mathrm{II}(C^{2, \alpha}(\bbS^{n-1}))\), define
    \begin{equation}
        \|(h_\mathrm{I}, \calA, h_\mathrm{II})\| = \|h_\mathrm{I}\|_{C^{2, \alpha}} + \|\calA\| + \|h_\mathrm{II}\|_{C^{2, \alpha}},
    \end{equation}
    and let \(\calB_\kappa^\alpha\) denote the ball of radius \(\kappa r_\veps^2\) in \(C^{2, \alpha}(\partial B_{r_0}) \times \mathbb{R}^{2n + 2} \times \pi_\mathrm{II}(C^{2, \alpha}(\bbS^{n-1}))\).
\end{defn}
\begin{defn}
    We define the conglomerate Cauchy data map
    \begin{equation}
        \bfC_\veps: C^{2, \alpha}(\partial B_{r_0}) \times \mathbb{R}^{2n + 2} \times \pi_\mathrm{II}(C^{2, \alpha}(\bbS^{n-1})) \rightarrow C^{1, \alpha}(\partial B_{r_0}) \times C^{2, \alpha}(\bbS^{n-1}) \times C^{1, \alpha}(\bbS^{n-1})
    \end{equation}
    to be the operator
    \begin{equation}
        \bfC_\veps(h_\mathrm{I}, \calA, h_\mathrm{II}) = (\calU_\veps(h_\mathrm{I}, \calA, h_\mathrm{II}), \calT_\veps(h_\mathrm{I}, \calA, h_\mathrm{II}) - \calS_\veps(h_\mathrm{II})).
    \end{equation}
    We remark that if \(\bfC_\veps(h_\mathrm{I}, \calA, h_\mathrm{II}) = 0\), then the boundary data \(h_\mathrm{I}, \calA, h_\mathrm{II}\) proves that the surface
    \begin{equation}
        \Sigma^* = \Sigma_{r_0}(h_\mathrm{I}) \cup \Sigma^c_{r_0, \veps}(h_\mathrm{I}, \calA, h_\mathrm{II}) \cup C_\veps(h_\mathrm{II})
    \end{equation}
    is a \(C^{1, \alpha}\) minimal surface, and thus a smooth minimal surface (as we desire to show).
\end{defn}
Just as before, we also have a simple conglomerate Cauchy data map:
\begin{defn}
    We define the simple conglomerate Cauchy data map
    \begin{equation}
        \bfC_0: C^{2, \alpha}(\partial B_{r_0}) \times \mathbb{R}^{2n + 2} \times \pi_\mathrm{II}(C^{2, \alpha}(\bbS^{n-1})) \rightarrow C^{1, \alpha}(\partial B_{r_0}) \times C^{2, \alpha}(\bbS^{n-1}) \times C^{1, \alpha}(\bbS^{n-1})
    \end{equation}
    to be the operator
    \begin{equation}
        \bfC_0(h_\mathrm{I}, \calA, h_\mathrm{II}) = (\calU_0(h_\mathrm{I}), \calT_0(\calA, h_\mathrm{II}) - \calS_0(h_\mathrm{II})).
    \end{equation}
\end{defn}
\begin{lem}\label{prop:the simple conglomerate C0 is an isomorphism}
    Letting
    \begin{equation}
        w_\calA^0 (r\theta) = \frac{e}{n-2}r^{2-n} + d + rR\cdot \theta + \veps r^{1-n}T\cdot \theta,
    \end{equation}
    we may write
    \begin{equation}
        \bfC_0(h_\mathrm{I}, \calA, h_\mathrm{II}) = (\calU_0(h_\mathrm{II}), w^0_\calA(r_\veps \cdot), r_\veps \partial_rw_\calA^0(r_\veps\cdot) - 2D_\theta h_\mathrm{II}).
    \end{equation}
    Furthermore, \(\bfC_0\) is an isomorphism from its domain onto its image, i.e., it is an isomorphism
    \begin{equation}
        \bfC_0 : C^{2, \alpha}(\partial B_{r_0}) \times \mathbb{R}^{2n + 2} \times \pi_\mathrm{II}(C^{2, \alpha}(\bbS^{n-1})) \rightarrow C^{1, \alpha}(\partial B_{r_0}) \times \operatorname{span}\{e_0, \ldots, e_n\} \times C^{1, \alpha}(\bbS^{n-1})
    \end{equation}
    where \(e_0, \ldots, e_n\) are the first \(n\) eigenfunctions of the Laplacian \(\Delta_{\bbS^{n-1}}\). The norm of \(\bfC_0^{-1}\) is bounded independently of \(\veps \in (0, 1)\).
\end{lem}
\begin{proof}
    The first claim follows from the definitions of \(\calT_0\) and \(\calS_0\) and the parametrization shown in Proposition \ref{prop:6.3}. The claim that \(\bfC_0\) is an isomorphism follows from Proposition \ref{prop:bounding Uveps - U0}, assumption (A.4), and the fact that the operator \(\Delta_0\) defined while discussing the catenoid is invertible.

    The fact that \(\|\bfC_0^{-1}\|\) is bounded uniformly for \(\veps \in (0, 1)\) (and so has a bound independent of \(\veps \in (0, 1)\)) follows from the expression in terms of \(w_\calA^0\). Indeed \(c\|\calA\| \le \|w_\calA^0\|_{1, \alpha} \le C\|\calA\|\) for constants \(c, C\) independent of \(\veps \in (0, 1)\), simply by the definition of \(\|\calA\|\).
\end{proof}

We remark that this is the place where it matters that the Green's function is asymptotically similar to a catenoid. This eliminated the term \(\frac{\veps}{n-2}r^{2 - n}\) from our expression for \(\bfC_0\) above, allowing us to produce the uniform bound on \(\|\bfC_0^{-1}\|\).

\subsubsection{Finding a zero of \(\bfC_\veps\)}
To find a zero for \(\bfC_\veps\), we find a fixed point for the operator \(\bfC_0^{-1}(\bfC_0 - \bfC_\veps)\).

\begin{lem}\label{lem:C0Cveps maps a ball into a ball}
    There exists a \(\kappa_0 > 0\) so that for all \(\kappa > 0\) there exists \(\veps_0 = \veps_0(\kappa) > 0\) such that
    \begin{equation}
        \bfC_0^{-1}(\bfC_0 - \bfC_\veps)(\calB_\kappa^\alpha) \subset \calB_{\kappa_0}^\alpha
    \end{equation}
    when \(\veps < \veps_0\).
\end{lem}
\begin{proof}
    First, \(\|\bfC_0^{-1}\|\) is bounded by a constant independent of \(\veps\) (and of course of \(\kappa\) as well), and so it suffices to bound \(\|\bfC_0 - \bfC_\veps\|\). In this lane, we have
    \begin{equation}
        \begin{split}
            \|\bfC_0 - \bfC_\veps\| & \le \|\calU_\veps - \calU_0\| + \|\calT_\veps - \calT_0\| + \|\calS_\veps - \calS_0\| 
            \\ & \le c_1(\|h_{\mathrm{I}}\|^2 + r_\veps^{n - 2/3}) + c_2r_\veps^2 + c_3r_\veps^2
            \\ & \le cr_\veps^2 + r_\veps^{n - 2/3} + \kappa^2r_\veps^4
            \\ & \le \kappa_0'r_\veps^2
        \end{split}
    \end{equation}
    where we have taken \(\kappa_0' = 2(c + 1)\) and chosen \(\veps_0\) small enough that \(\kappa^2r_\veps^4 < cr_\veps^2\) for \(\veps < \veps_0\). We note that \(c_1, c_2, c_3\) do not depend on \(\kappa\) whatsoever, as we saw when deriving the Cauchy data estimates in the relevant propositions. Finally, we clearly can take \(\kappa_0 = \|\bfC_0^{-1}\|\kappa_0'\).
\end{proof}
\begin{prop}\label{prop:fixed-point}
    For \(\kappa > \kappa_0\) and \(\veps < \veps(\kappa)\) (as produced in the last Lemma), the mapping
    \begin{equation}
        \bfC_0^{-1}(\bfC_0 - \bfC_\veps): \calB_{\kappa}^\alpha \rightarrow \calB_{\kappa_0}^\alpha
    \end{equation}
    has a fixed point.
\end{prop}
\begin{proof}
    The proof of this proposition is the same as the proof of Proposition 9.1 in \cite{Fakhi-Pacard}, using Schauder's fixed point theorem on a family of smoothed operators \(\bfD^q\bfC_0^{-1}(\bfC_0 - \bfC_\veps)\), \(q \in \bbN\).
\end{proof}

Finally, we show that this new surface $\Sigma_\veps:=\Sigma_{r_0}(h)\cup \Sigma^c_r(h,h_{\mathrm{II}},\calA_{\mathrm{II}})\cup C_{\veps}(h_{\mathrm{II}})$ is non-degenerate for all $\veps>0$ small enough. 

\begin{lem}\label{lem: nondegeneracy}
For all $\veps>0$ small enough, the hypersurface $\Sigma_\veps$ is non-degenerate in the sense of Lemma \ref{lem:invertability-of-L_M} assuming that $\Sigma$ is non-degenerate.
\end{lem}
\begin{proof}
We argue by contradiction. Thus suppose that there exists a sequence $\veps_i\to 0$ such that the operator $\calL_{\Sigma{{\veps_i}}}$ is not injective on $\calE^{2,\alp}_\delta(\Sigma_{\veps_i})$ for some $\delta\in (-\frac{n+2}{2},-\frac{n}{2}).$ Consequently, there exists a non-zero function $f_i\in \calE^{2,\alp}_\delta(\Sigma_{\veps_i})$ such that 
\EQ{
\calL_{\Sigma{{\veps_i}}} f_i = 0.
}
Since by construction we have that
\EQ{
\Sigma_{\veps_i} &\equiv \Sigma_{r_0}(h_i)\cup \Sigma^c_{r_0,\veps_i}(h_i,h_{\mathrm{II},k},\calA_{\mathrm{II},k})\cup C_{\veps_i}(h_{\mathrm{II},k})\\
&= \Sigma_{k}^{c}\cup \bigcup_{j=1}^k E_{k,j} \cup \Sigma^c_{r_0,\veps_i}(h_i,h_{\mathrm{II},k},\calA_{\mathrm{II},k})\cup C_{\veps_i}(h_{\mathrm{II},k})
}
where we expressed $\Sigma_{r_0}(h_i)$ as union of a compact piece $\Sigma_{k}^{c}$ and $k$ planar ends $\{E_{i, j}\}_{j=1}^{k}$. Define on $\Sigma_{\veps_i}$ some weight function $g_i>0$, as follows:
\EQ{
g_i &\simeq 1 \text{ on } \Sigma_{i}^{c},\quad g_i \simeq e^{\delta s} \text{ on } \bigcup_{j=1}^k E_{k,j}\quad g_i \simeq r^{-\delta} \text{ in } \Sigma^c_{r_{0}, \varepsilon_{i}}\left(h_{\varepsilon_{i}}, \mathcal{A}_{i}, h_{\mathrm{II}, i}\right)\\
g_i &\simeq r_{\varepsilon_{i}}^{-\delta} e^{\delta (s-s_{\varepsilon_{i}})} \text{ in } C_{\varepsilon_{i}}(h_{\mathrm{II}, i})
}
where $f \simeq g$ means that $C^{-1} \leq f / g \leq C$ for some fixed constant $C>0$. Define, a weight $g_\infty$ on $\Sigma$ such that 
\EQ{
g_\infty>0, \quad g_\infty=e^{\delta s} \text{ on } \bigcup_{j=1}^k E_{k,j}, \quad g_i \to g_\infty \text{ in }C^{2,\alp}_{\text{loc}}(\Sigma \setminus \{0\}),\text{ and }  \sup _{\Sigma_{\veps_i}} g_i^{-1} f_{i}=1.
}
Since the indicial roots of $\mathcal{L}_{\Sigma_{\veps_i}}$ at each end are given by $\pm \gamma_{j}$, any bounded solution of $\mathcal{L}_{\Sigma_{\veps_i}} w=0$ which belongs to the space $\mathcal{E}_{\delta}^{2, \alpha}\left(\Sigma_{\veps_i}\right)$ decays like $e^{-\frac{n+2}{2} s}$ at each end. This implies that the above supremum is achieved at some point $p_{i} \in \Sigma_{\veps_i}$. Either (up to a subsequence) $p_i\to p_\infty\in \Sigma \setminus\{0\}$, $|p_i|\to 0$, $|p_i|\to \infty$ such that $p_i \in E_j$ for some fixed $j=1,\ldots,k$ or $|p_i|\in C_{\veps_i}(h_{\mathrm{II},i})$. We analyze each of these cases separately. 
\paragraph{Case 1}
If up to a subsequence $p_{i}\to p_{\infty}$ for some $p_{\infty} \in \Sigma \backslash\{0\}$ then there exists a subsequence $f_{i}\to f_\infty$ in $C^{2,\alp}_{\text{loc}}(\Sigma \backslash\{0\})$, with $f_\infty \leq C g_\infty$ and consequently solves:
$$
\mathcal{L}_{\Sigma} f_{\infty}=0 .
$$
Thus $f_\infty$ inherits the decay at the origin from $g_\infty$, making the singularity at the origin removable, implying that $f_\infty$ is a global solution, which contradicts the fact that $\Sigma$ is non-degenerate. 
\paragraph{Case 2} If $|p_i|\to \infty$, then $p_i \in E_{i,j}$ for some fixed $j=1,\ldots,k$. Then if we define the rescaled function $f_i$ on the end with $s_i\to \infty$ such that
\EQ{
\bar{f}_{i}(s, \theta) \equiv e^{-\delta s_{i}} f_{i}\left(s+s_{i}, \theta\right),
}
we deduce that $\bar{f}_i\to f_\infty$ in $C^{2,\alp}_{\text{loc}}(\R\times \bbS^{b-1})$ such that
\EQ{
\Delta_{0} f_\infty=0,
}
such that $f_\infty \leq C e^{\delta s}$ for some constant $C>0$ which contradicts the order $-\frac{n+2}{2}$ order decay since $\delta \in\left(-\frac{n+2}{2},-\frac{n}{2}\right)$.
\paragraph{Case 3} Finally consider the case when $|p_i|\in C_{\veps_i}(h_{\mathrm{II},i})$ or $|p_i|\to 0$. Here we consider two possible subcases. The sequence $p_{i}$ remains in the annular region $\Sigma^c_{r_{0}, \varepsilon_{i}}\left(h_{i}, \mathcal{A}_{i}, h_{\mathrm{II}, i}\right)$  or the sequence $p_{i}$ belongs to the truncated $n$-catenoid $C_{\varepsilon_{i}}\left(h_{\mathrm{II}, i}\right)$. In either case, the point $p_i$ can be identified with parameters $(s_i,\theta_i)$ along the catenoid $C_{\veps_i}$ since $\Sigma^c_{r_{0}, \varepsilon_{i}}\left(h_{i}, \mathcal{A}_{i}, h_{\mathrm{II}, i}\right)$ is a normal graph over $C_{\veps_i}$, the only difference is that if $s_i<s_{\veps_i}$ then $p_i$ lives in the annular region while $s_i>s_{\veps_i}$ implies that $p_i$ lives in the truncated catenoid $C_{\veps_i}(h_{\mathrm{II}, i}).$ Now recall that $g_i \simeq r_{\varepsilon_{i}}^{-\delta} e^{\delta (s-s_{\varepsilon_{i}})} \text{ in } C_{\varepsilon_{i}}(h_{\mathrm{II}, i})$ and therefore if we rescale the function $f_i$ as
\EQ{
\bar{f}_{i}(s, \theta) \equiv r_{\varepsilon_{i}}^{\delta} e^{\delta\left(s_{\varepsilon_{i}}-s_{i}\right)} f_{i}\left(s+s_{i}, \theta\right).
}
then this sequence converges uniformly on compact sets of $\mathbb{R} \times S^{n-1},$ to a non-trivial function $f_\infty$ such that either
\EQ{
\Delta_{0} f_\infty=0.
}
when $\left|s_{i}\right|$ tends to $\infty$, or 
\EQ{
\mathcal{L} f_\infty=0 \quad 
}
when $s_{i}\to s_*\in \R$. Furthermore, $f_\infty\leq C e^{\delta s}$ which is not possible for the range of $\delta \in (-\frac{n+2}{2},-\frac{n}{2}).$

Since we have obtained a contradiction in all the possible cases, we conclude that $\Sigma_{\veps}$ is non-degenerate. 
\end{proof}



Combining all the previous results in this section we deduce,
\begin{thm}\label{thm:general statement of gluing}
Let \(\Sigma \subset \mathbb{R}^{n+1}\), \(n \ge 3\), be a nondegenerate minimal surface with finite total curvature and a planar end \(E\) which is graphical at infinity. Then there exists a sufficiently small \(\veps_0 > 0\) such that for any \(\veps < \veps_0\), we can glue a half-catenoid \(C_\veps\) to \(\Sigma\) at a point \(x \in E\) to produce a nondegenerate \(\Sigma'\) such that the asymptotic plane of \(C_\veps\) is parallel to the asymptotic plane of \(E\).
\end{thm}
\begin{proof}
We must choose the gluing location and the scale \(\veps\). We choose \(\veps\) sufficiently small so that \(\Sigma'\) is nondegenerate (see Lemmas \ref{lem: nondegeneracy} and \ref{lem:invertability-of-L_M}).
    
From our assumptions, we have that \(E\) is a graph of a function \(u\) over a sufficiently large annulus. We choose a point \(x\) such that \(|\nabla u(x)| < r_\veps\), where \(r_\veps\) is defined as in \ref{defn:r_veps and s_veps}. It is possible to do this, because the end \(E\) is planar. It follows that we can perform the gluing construction at \(x\) to produce a surface \(\Sigma'\). The end of the catenoid \(C_\veps\) we glue on will be parallel to the asymptotic plane of the end \(E\) by the construction process.
\end{proof}

\subsection{Proof of Theorems \ref{thm:nonproper} and \ref{thm:k-ends}}
In order to prove Theorems \ref{thm:nonproper} and \ref{thm:k-ends}, we will construct a sequence of surfaces \(\Sigma_k\) which are complete, embedded, and have \(k\) planar ends. We will construct \(\Sigma_k\) so that their planar ends converge to a limit plane, which will give us the improperness of the limit surface \(\Sigma_\infty\).

\begin{lem}\label{lem: limit stuff}
    There exists a sequence of complete, embedded minimal surfaces \(\Sigma_k\) with the following properties:
    \begin{enumerate}
        \item \(\Sigma_k\) has \(k\) parallel planar ends.
        \item There exists an open slab \(\Pi = \{0 < x_{n + 1} < C_\text{max}\}\) such that \(\Sigma_k \subset \Pi\) for all \(k\). Furthermore, if \(E_k\) is the ``top'' end of \(\Sigma_k\) and \(P_k\) is the asymptotic plane of \(E_k\), then \(E_k \rightarrow P_k\) uniformly at infinity, and \(P_k \rightarrow \{x_{n+1} = C_\text{max}\}\) as \(k \rightarrow \infty\).
        \item There exists a sequence of rectangular neighborhoods \(N_j\) such that
        \begin{equation}
            \sup_{x \in \Sigma_k \setminus \bigcup_{j = 1}^\infty N_j}|A|(x) < 1
        \end{equation}
        when \(j \ge k\), and an increasing sequence of constants \(c_j \nearrow \infty\) such that
        \begin{equation}
            \sup_{x \in \Sigma_k \cap N_j} |A|(x) < c_j
        \end{equation}
        when \(j \le k\).
    \end{enumerate}
    Finally, the neighborhoods \(N_j\) can be chosen so that
    \begin{equation}
            \overline{N_j} \subset \Pi.
        \end{equation}
    I.e., \(\overline{N_j}\) is a positive distance from \(\partial \Pi\).
\end{lem}

The idea is that \(\{\Sigma_k\}\) will be a sequence of minimal surfaces we get by ``stacking'' catenoids. i.e., \(\Sigma_{k + 1}\) is \(\Sigma_k\) with a half catenoid glued onto it somewhere. The neighborhood \(N_j\) contains the \(j\)th neck of this construction. The curvature bounds represent how the curvature of \(\Sigma_k\) is small away from the necks and large near the necks. However, it is still locally bounded near the necks.

\begin{proof}
    Let \(\Sigma_1\) be the catenoid \(C_1\), centered at the origin, oriented with the vertical axis, and scaled such that \(C_1 \cap \{x_{n+1} = 0\}\) is a sphere of radius 1.

    Clearly, \(\Sigma_1\) satisfies properties (1), (2), and (3). Now, suppose we have constructed the surfaces \(\Sigma_k\) and the neighborhoods \(N_k\). We will now glue a half catenoid \(C_\veps\) to \(\Sigma_k\) to produce \(\Sigma_{k+1}\). We must choose the gluing location and the scale \(\veps\). We choose \(\veps < \min\{2^{-k}, \veps_0^k\}\), where \(\veps_0^k\) is the \(\veps_0\) from Theorem \ref{thm:general statement of gluing}. To choose the gluing location, we choose a point far enough away that the perturbation in \(N_j\) for \(j \le k\) does not change the curvature bounds (i.e., far enough away that the necks do not get perturbed much), and also far enough away that it is possible to glue on a parallel end (as in Theorem \ref{thm:general statement of gluing}). By the construction, the new end \(E_{k+1}\) is a (almost) normal perturbation of \(C_\veps\) by a function \(w\) with a bound \(\|w\|_{2, \alpha, \delta} < \infty\). This bound gives us the uniform convergence to a plane \(P_{k + 1}\) in property (2). To ensure that \(\Sigma_{k + 1}\) is embedded, we choose the gluing location far enough away that \(E_{k + 1}\) is almost flat when it is near \(N_j\) for \(j \le k\). We know from the gluing construction that the perturbations \(p(\Sigma_k) = \Sigma_{r_0}(h_\mathrm{I}) \cup \Sigma_{r_0}^c(h_\mathrm{I}, \calA, h_\mathrm{II})\) and \(p(C_\veps) = C_\veps(h_\mathrm{II})\) are each embedded (as each is a graph over \(\Sigma_k\) (minus a disk) and \(C_\veps\), respectively). So \(\Sigma_{k+1}\) will fail to be embedded if points in \(p(C_\veps)\) intersect points in \(p(\Sigma_k)\) away from the gluing boundary.

    Suppose that this does happen at some point \(x \in p(C_\veps)\). The intersection must be transverse by the maximum principle. Thus, there is a subset \(U\) of \(p(C_\veps)\) underneath \(p(\Sigma_k)\). The connected component \(U_c\) of this subset containing \(x\) does not contain the gluing boundary. Furthermore, since \(p(C_\veps)\) converges uniformly to its asymptotic plane, \(U_c\) is bounded, hence \(\overline{U_c}\) is compact. Thus the height function \(x_{n+1}\) will have a local minimum in \(\overline{U_c}\).
    
    Because every intersection of \(p(C_\veps)\) and \(p(\Sigma_k)\) away from the gluing boundary is transverse, the minimum of the height function cannot occur on \(\partial U_c\), as there will always be a direction to decrease \(x_{n + 1}\) in the interior of \(U_c\). Thus, \(x_{n + 1}\) has a minimum in \(U_c\), which is a contradiction because \(x_{n + 1}\) restricted to \(\Sigma_{k+1}\) is a harmonic function on \(\Sigma_{k + 1}\).
    
    We choose \(N_{k + 1}\) to be a rectangle containing the neck of the catenoid \(C_\veps\) we just glued on. To show property (2) holds, let \(\veps_k\) be the \(\veps\) we chose when gluing a \(C_{\veps}\) to \(\Sigma_k\). The limit planes will be contained within a slab \(\{-1 < x_{n+1} < \sum \veps_k < \infty\}\), since \(\veps_k < 2^{-k}\). Thus, there is a smallest slab containing all \(\Sigma_k\). Once we identify this slab, we can redefine \(N_k\) to fit into this slab as specified in the theorem statement.
\end{proof}

\begin{proof}[Proof of Theorem \ref{thm:k-ends}]
We proceed by induction. When $k=2$, the usual catenoid $C_1$ yields an example of a non-degenerate complete, embedded, minimal hypersurface with two planar ends since $n\geq 3$. Suppose we have constructed a non-degenerate complete minimal embedded hypersurface $\Sigma_k^n\subset \R^{n+1}$ with $k$ planar ends. Then by Theorem \ref{thm:general statement of gluing}, we can construct a non-degenerate complete minimal hypersurface $\Sigma_{k+1}^n\subset \R^{n+1}$ with $k+1$ planar ends by gluing a rescaled half-catenoid far along the $k$th planar end. To show that $\Sigma_{k+1}^n$ is embedded we can argue as in the proof of Lemma \ref{lem: limit stuff}. 
\end{proof}

\begin{proof}[Proof of Theorem \ref{thm:nonproper}]
    We show \(\Sigma_\infty\) exists by finding a convergent subsequence of the \(\Sigma_k\) that converges locally smoothly. Let \(\{Q_i\}\) be a countable set of rectangles which form an open cover of \(\Pi - \bigcup_{j = 1}^\infty N_k\) such that \(\overline{Q_i} \subset \Pi\) for all \(i\). Then \(\{R_i\} = \{Q_1, N_1, Q_2, N_2, Q_3, N_3, \ldots\}\) is an open cover of \(\Pi\) such that \(\overline{R_i} \subset \Pi\) for each \(i\). On \(R_i\), we have that 
    \begin{equation}
        \sup_{x \in \Sigma_k \cap R_i} |A| \le \min\{1, c_{i}\}.
    \end{equation}
    Thus, \(\Sigma_k \cap R_i\) is a sequence of minimal surfaces with bounded curvature, and thus there is a subsequence which converges smoothly on \(R_i\). It now follows from an Arzela-Ascoli style diagonalization argument that there is a subsequence of \(\Sigma_k\) which converges locally smoothly on all of \(\Pi\) to a surface \(\Sigma_\infty \subset \Pi\). We know \(\Sigma_\infty \subset \Pi\) because \(\overline{R_i} \subset \Pi\).

    Since each \(\Sigma_k\) is complete and embedded, \(\Sigma_\infty\) is also complete and embedded. We now show that \(\Sigma_\infty\) is improper. Each perturbation to produce \(\Sigma_{k + 1}\) from \(\Sigma_k\) left the asymptotic planes of \(\Sigma_{k + 1}\) unchanged. That is, if \(P_1, \ldots, P_k\) are the asymptotic planes of \(\Sigma_k\), then these are also asymptotic planes of \(\Sigma_{k + 1}\). Thus, \(\Sigma_\infty\) has infinitely asymptotic planes \(P_1, \ldots, P_n ,\ldots\). These planes converge to \(\partial^+ \Pi\). Thus, \(\overline{\Sigma_\infty} = \Sigma_\infty \cup \partial^+ \Pi\), and \(\Sigma_\infty\) is improper.
\end{proof}

\bibliographystyle{alpha}
\bibliography{references}
\bigskip
\centerline{\scshape Shrey Aryan}
\smallskip
{\footnotesize
 \centerline{Department of Mathematics, Massachusetts Institute of Technology}
\centerline{77 Massachusetts Avenue
Cambridge, MA 02139-4307, USA}
\centerline{\email{shrey183@mit.edu}}
} 
\medskip
\centerline{\scshape Alexander D. McWeeney}
\smallskip
{\footnotesize
 \centerline{Department of Mathematics, Massachusetts Institute of Technology}
\centerline{77 Massachusetts Avenue
Cambridge, MA 02139-4307, USA}
\centerline{\email{alexmcw4@mit.edu}}
}

\end{document}